\documentclass[11pt,a4paper]{article}
\usepackage[T1]{fontenc}
\usepackage[utf8]{inputenc}
\usepackage{lmodern}
\usepackage[margin=25mm]{geometry}
\usepackage{amsmath,amssymb,amsthm,mathrsfs}
\usepackage{enumitem}
\usepackage[numbers,sort&compress]{natbib}
\usepackage{microtype}
\usepackage[hidelinks]{hyperref}
\setlist[enumerate]{leftmargin=.38in}
\allowdisplaybreaks[1]
\hypersetup{
  pdftitle={Classical Solutions for a Finite-Horizon Exit-Time Problem with Degenerate Diffusion Control},
  pdfauthor={Lihua Bai and Linyu Miao}
}

\theoremstyle{plain}
\newtheorem{Theorem}{Theorem}[section]
\newtheorem{Proposition}[Theorem]{Proposition}
\newtheorem{Lemma}[Theorem]{Lemma}
\newtheorem{Corollary}[Theorem]{Corollary}
\theoremstyle{definition}
\newtheorem{Assumption}{Assumption}

\newtheorem{Example}[Theorem]{Example}
\newtheorem{Remark}[Theorem]{Remark}

\title{Classical Solutions for a Finite-Horizon Exit-Time Problem with Degenerate Diffusion Control}
\author{Lihua Bai \qquad Linyu Miao\\[0.6em]
\small School of Mathematical Sciences, Nankai University\\
\small Tianjin 300071, China\\[0.3em]
\small \href{mailto:lhbai@nankai.edu.cn}{\texttt{lhbai@nankai.edu.cn}}
\quad \href{mailto:mly202210@126.com}{\texttt{mly202210@126.com}}}
\date{}

\begin{document}
\maketitle

\begin{abstract}
We study finite-horizon exit-time control of a one-dimensional affine
diffusion, with an unbounded control acting on both drift and volatility.
Since an admissible control can cancel the instantaneous volatility, the
associated HJB equation is not uniformly parabolic. Moreover, the location
of the degeneracy is not prescribed by the model coefficients but depends
on the unknown value function and optimal feedback. Rather than relying
on a viscosity-solution formulation, we work under suitable structural
assumptions and prove that the value function is a $C^{1,2}$ classical
solution throughout the interior of the state domain and is strictly
convex in the state variable. Moreover, we construct a locally Lipschitz
optimal feedback, explicitly identify the degenerate set as a curve, and
show that the closed-loop drift on this curve equals the time derivative
of its state coordinate. Finally, we prove invariance and inaccessibility
before exit and derive exact squared-logarithmic asymptotics for approach
probabilities.
\end{abstract}

{\small
\noindent\textbf{Keywords.} Control-dependent diffusion; Degenerate HJB equation; classical solution; finite-horizon exit-time control.\par
\noindent\textbf{Mathematics Subject Classification (2020).} 60H30, 93E20, 35K65.
}

\section{Introduction}
We study a finite-horizon exit-time control problem for the one-dimensional
affine diffusion
\[
dY_s=(a_sY_s+b_su_s+c_s)\,ds
+(\sigma_su_s+\gamma_s)\,dW_s,
\qquad s\in[t,T],\qquad Y_t=y,
\]
evolving in a time-dependent interval.  Define the value function 
\[
V(t,y):=\inf_u\mathbb E\!\left[
h(Y_T-d)\mathbf 1_{\{\tau=T\}}
+g(\tau,Y_\tau)\mathbf 1_{\{\tau<T\}}\right].
\]
Here $u$ ranges over admissible controls taking values in $\mathbb R$,
$\tau$ is the first exit time from the interval capped by $T$, $g$ is the prescribed lateral exit cost, and $h$ is a convex terminal loss
centered at the target level $d$.

Finite-horizon exit-time control is a well-established class of stochastic
control problems whose HJB equations carry both terminal and lateral boundary
data. The regularity of the associated HJB
problem depends strongly on parabolicity. When the covariance matrix admits a uniform positive lower bound independent
of the admissible control--that is, when the uniform parabolicity condition
holds--classical Cauchy--Dirichlet theory can identify the value function with
a classical solution of the HJB equation and justify an optimal Markov
feedback, under the usual regularity and compatibility assumptions
\cite{FlemingSoner2006}.  This applies both when the diffusion coefficient is
fixed and when it depends on the control but retains a uniform lower bound.
More specific finite-horizon exit results include a strong HJB
solution with feedback controls on half-spaces
\cite{Calvia2022HJBEA} and a classical solution for a half-line
exit problem \cite{Zawisza2025StochasticEC}.  Related actuarial problems
stopped at ruin have yielded classical value functions in specific
one-dimensional finite-horizon dividend, reinsurance, and liquidity models
by exploiting their variational-inequality or free-boundary structure
\cite{DeAngelisEkstrom2017,GuanXuZhou2023,ChenGuanYi2021}.

When the uniform parabolicity required by classical theory is
unavailable, viscosity solutions provide the robust framework for
finite-horizon exit-time control problems. Dynamic programming and continuity
results for possibly degenerate controlled diffusions in time--space domains
are developed in
\cite{DongKrylov2007,Rokhlin2013DPP,BayraktarSongYang2011}, while stochastic
representations and viscosity uniqueness extend to degenerate nonlocal
terminal--boundary Bellman equations in \cite{GongMouSwiech2019}.
Complementary comparison, boundary-regularity, and approximation results for
parabolic Cauchy--Dirichlet HJB equations are given in
\cite{Chaumont2004,PicarelliReisingerRotaetxe2020}. At the application level,
viscosity characterizations have been obtained for
insurance surplus control stopped at ruin, finite-time liquidation with
control-dependent drift and diffusion, and stochastic reach--avoid problems
with discontinuous exit rewards
\cite{Zhu2011RiskControl,BianWuZheng2016,
	MohajerinEsfahaniChatterjeeLygeros2016}.

The present model is plainly not uniformly parabolic. In contrast to the
viscosity-solution literature reviewed above, we seek a finer characterization
of the value function and the optimal control.  The main difficulty is already
visible in the associated HJB equation
\[
V_t+(a_ty+c_t)V_y
+\inf_{u\in\mathbb R}
\left\{
b_tuV_y+\frac12(\sigma_tu+\gamma_t)^2V_{yy}
\right\}=0.
\]
If $V_{yy}<0$, the infimum is $-\infty$; if $V_{yy}=0$, it is finite only
when $V_y=0$.  Thus finiteness of the Hamiltonian requires
$V_{yy}\ge0$ and, wherever $V_{yy}=0$, also $V_y=0$.  At points where
$V_{yy}>0$, formal minimization gives
\[
u^*(t,y)
=-\frac{\gamma_t}{\sigma_t}
-\frac{b_t}{\sigma_t^2}\frac{V_y(t,y)}{V_{yy}(t,y)},
\qquad
\sigma_tu^*(t,y)+\gamma_t
=-\frac{b_t}{\sigma_t}\frac{V_y(t,y)}{V_{yy}(t,y)},
\]
and hence the reduced HJB equation
\[
V_t+
\left(a_ty+c_t-\frac{b_t\gamma_t}{\sigma_t}\right)V_y
-\frac12\left(\frac{b_t}{\sigma_t}\right)^2
\frac{V_y^2}{V_{yy}}=0.
\]
The location of any zero-volatility set is therefore determined by the unknown value
function, but the reduced equation becomes singular when $V_y$ and $V_{yy}$
vanish together.  At such a point $V_y/V_{yy}$ has the indeterminate form
$0/0$, while the original Hamiltonian does not distinguish among admissible
control values.  Consequently, neither the limiting feedback nor the
closed-loop drift and volatility on the curve can be obtained pointwise from
the HJB equation.

To focus on the essential difficulty, we first apply a deterministic affine
change of state and a shift of control, reducing the model to the homogeneous
dynamics $dX_s=b_su_s\,ds+\sigma_su_s\,dW_s$.  A comparison estimate singles
out $x=0$ and motivates two one-sided problems.  On each side, the formal
Legendre dual is constructed as an optimal-stopping value for geometric
Brownian motion.  A half-line heat-kernel estimate shows that the correction
caused by the moving state boundary is negligible to every algebraic order
near $p=0$, including
the first three derivatives needed for inversion.  Strict dual convexity then
permits inverse-Legendre reconstruction.

Our first main result shows that the interior zero-volatility set is precisely
an explicitly identified curve and that the closed-loop drift on this curve
equals the time derivative of its state coordinate.  Our second result
constructs a classical solution of the HJB equation that belongs to
$C^{1,2}$ throughout the interior of the state domain.  The formal minimizing
feedback extends across the curve to a continuous, locally Lipschitz
feedback, and the resulting closed-loop control is admissible and optimal.
Consequently, the constructed solution coincides with the value function.
The value function is strictly convex in the state variable, even though its
second derivative may vanish on the curve when the terminal minimum is flat.
Our third result shows that a closed-loop trajectory starting on the curve
remains on it, whereas one starting away from the curve cannot reach it before
exiting the state interval.  We also derive the exact squared-logarithmic
asymptotic rate of the probability that, at a fixed later time and before
exit, the optimally controlled state lies within a small distance of the
curve.

Section~\ref{sec:model} formulates the model and assumptions.
Section~\ref{sec:candidate-interface} identifies the candidate zero-volatility
set and the two one-sided problems. Section~\ref{sec:auxiliary} constructs the
dual stopping problem and analyzes the solution near $p=0$.
Section~\ref{sec:global} reconstructs the global classical solution;
Section~\ref{sec:affine} returns to the affine model, and
Section~\ref{sec:probability} determines whether the optimal state can reach
or remain on the zero-volatility curve and derives the corresponding approach
rate. Supplementary results and technical arguments are collected in the
appendices.

\section{Problem formulation}
\label{sec:model}

We use the standard notation for parabolic H\"older and Sobolev spaces. In
particular, $C^{1,2}$ denotes functions with one continuous time derivative
and two continuous spatial derivatives,
while $W_\rho^{2,1}(Q)$ consists of functions
$f,f_x,f_{xx},f_t\in L^\rho(Q)$, with derivatives understood
distributionally. Local spaces carry the subscript ``$\mathrm{loc}$'', and
equations for Sobolev functions are understood almost everywhere.

\subsection{General formulation}

Let $(\Omega,\mathcal F,\mathbb F,\mathbb P)$ be a filtered probability space
satisfying the usual conditions and supporting a one-dimensional standard
Brownian motion $W=(W_s)_{0\le s\le T}$, where $T>0$ is fixed. For an initial
pair $(t,y)\in[0,T]\times\mathbb R$, consider the controlled linear diffusion
\begin{equation}
	\label{eq:general-state}
	dY_s=(a_sY_s+b_su_s+c_s)\,ds+(\sigma_s u_s+\gamma_s)\,dW_s,
	\qquad s\in[t,T],\qquad Y_t=y,
\end{equation}
where $a,c,\gamma\in C([0,T])$ and
$b,\sigma\in C^\alpha([0,T])$ are deterministic.  The control set is the
whole real line $U=\mathbb R$.

For a fixed initial pair $(t,y)$, an $\mathbb R$-valued process
$u=(u_s)_{s\in[t,T]}$ is called admissible if it is
$\mathbb F$-progressively measurable, satisfies
$\mathbb E\int_t^T|u_s|^2\,ds<\infty$, and the state equation
\eqref{eq:general-state} admits a unique strong solution. The corresponding
admissible class is denoted by $\mathcal U^{\rm gen}(t,y)$.

Let $\ell,r\in C^{1+\alpha}([0,T])$ satisfy
$\ell(t)<r(t)$ for every $t\in[0,T]$.  Define the moving time--space domain by
\[
\mathcal D:=\{(t,\xi)\in[0,T]\times\mathbb R:
\ell(t)\le \xi\le r(t)\}.
\]
We write
\[
\mathcal D^\circ
:=\{(t,x)\in[0,T)\times\mathbb R:\ell(t)<x<r(t)\}
\]
for the interior of the moving domain and
\[
\partial_p\mathcal D
:=
\bigl\{(t,\ell(t)):0\le t\le T\bigr\}
\cup
\bigl\{(t,r(t)):0\le t\le T\bigr\}
\cup
\bigl\{(T,x):\ell(T)\le x\le r(T)\bigr\}.
\]
This is the parabolic boundary of $\mathcal D$.

The prescribed lateral costs
$g_\ell,g_r\in C^{1+\alpha}([0,T])$ are strictly positive, and the lateral
payoff $g$ is defined by
\begin{equation}
	\label{eq:lateral-payoff}
	g(t,x):=
	\begin{cases}
		g_\ell(t), & x=\ell(t),\\
		g_r(t),    & x=r(t).
	\end{cases}
\end{equation}

Fix a target level $d\in\mathbb R$ and set
\[
I_T^d:=
\bigl[\min\{\ell(T)-d,0\},\max\{r(T)-d,0\}\bigr].
\]
The terminal loss $h\in C^2(I_T^d)$ is locally $C^3$ on each side of zero,
nonnegative, and convex.  We assume that zero is its unique minimizer,
\[
h(0)=h'(0)=0,
\qquad
h''(x)>0\quad\text{for }x\ne0,
\]
without imposing third-order regularity at zero.  The terminal and lateral
data are compatible at the corners:
\[
g_\ell(T)=h(\ell(T)-d),
\qquad
g_r(T)=h(r(T)-d).
\]

For $(t,y)\in\mathcal D$ and
$u\in\mathcal U^{\rm gen}(t,y)$, denote by
$\mathbb P_{t,y}^u$ the law of the controlled state starting from
$Y_t=y$, and by $\mathbb E_{t,y}^u$ the corresponding expectation.
Define its first exit time from the moving open interval by
\[
\tau^{t,y;u}
:=\inf\{s\in[t,T]:Y_s\notin(\ell(s),r(s))\}\wedge T,
\]
with the convention $\inf\varnothing=+\infty$.  Since the state process and
the boundary curves are continuous, on $\{\tau^{t,y;u}<T\}$ one has
\[
Y_{\tau^{t,y;u}}
\in\{\ell(\tau^{t,y;u}),r(\tau^{t,y;u})\},
\]
and hence the boundary payoff in \eqref{eq:lateral-payoff} is well defined.

The general exit-time criterion is
\begin{equation}
	\label{eq:general-cost}
	J^{\rm gen}(t,y;u):=
	\mathbb E_{t,y}^u\!\left[
	h(Y_T-d)\mathbf 1_{\{\tau^{t,y;u}=T\}}
	+g\bigl(\tau^{t,y;u},Y_{\tau^{t,y;u}}\bigr)
	\mathbf 1_{\{\tau^{t,y;u}<T\}}
	\right],
\end{equation}
and the corresponding value function is
\[
V^{\rm gen}(t,y):=
\inf_{u\in\mathcal U^{\rm gen}(t,y)}
J^{\rm gen}(t,y;u),
\qquad (t,y)\in\mathcal D.
\]
When the control is clear from the context, we suppress the superscript
$u$.  In particular, after the optimal feedback has been constructed,
$\mathbb P_{t,y}$ and $\mathbb E_{t,y}$ denote the probability law and
expectation of the corresponding optimal closed-loop state.

If the state reaches a lateral boundary exactly at time $T$, the convention in
\eqref{eq:general-cost} assigns the terminal cost. This convention is
consistent with the terminal-corner compatibility above.  That compatibility
alone does not ensure that the interior value attains the prescribed lateral
costs continuously.  Boundary discontinuity can occur in controlled exit
problems even on bounded domains; sufficient conditions based on boundary
regularity and controlled barriers are given in
\cite{FlemingSoner2006,BayraktarSongYang2011,DongKrylov2007}.  We exclude the
corresponding generalized-Dirichlet regime. Under
Assumption~\ref{ass:dual-comparison} below, the prescribed lateral values are
attained continuously, as proved through the dual construction and
inverse-Legendre reconstruction.

\subsection{The homogeneous model}

We first study the homogeneous specialization of the preceding problem, with
the target translated to zero; thus the preceding data conditions are used
with $d=0$:
\begin{equation}
	\label{eq:state}
	dX_s=b_su_s\,ds+\sigma_su_s\,dW_s,
	\qquad s\in[t,T],\qquad X_t=x.
\end{equation}
For $(t,x)\in\mathcal D$, let $\mathcal U(t,x)$ denote the corresponding
specialization of the admissible class $\mathcal U^{\rm gen}(t,x)$.

For $(t,x)\in\mathcal D$ and $u\in\mathcal U(t,x)$, let
\begin{equation*}
	\tau^{t,x;u}:=\inf\{s\in[t,T]:X_s\notin(\ell(s),r(s))\}\wedge T.
\end{equation*}
When no ambiguity arises, we suppress the superscripts $(t,x;u)$ and simply
write $\tau$. The cost functional associated with \eqref{eq:state} is
\begin{equation*}
	J(t,x;u):=
	\mathbb E_{t,x}^u\!\left[
	h(X_T)\mathbf 1_{\{\tau=T\}}
	+g(\tau,X_\tau)\mathbf 1_{\{\tau<T\}}
	\right],
\end{equation*}
where the lateral payoff $g$ is defined by \eqref{eq:lateral-payoff}. The value
function is
\begin{equation}
	\label{eq:value}
	V(t,x):=\inf_{u\in\mathcal U(t,x)}J(t,x;u),
	\qquad (t,x)\in\mathcal D.
\end{equation}

\begin{Assumption}[Controlled coefficients]
	\label{ass:coefficients}
	\[
	\inf_{0\le t\le T}|b_t|>0,
	\qquad
	\inf_{0\le t\le T}|\sigma_t|>0.
	\]
\end{Assumption}

A formal application of the dynamic programming principle gives the HJB
boundary--terminal problem
\begin{equation}
	\label{eq:HJB-boundary-problem}
	\left\{
	\begin{aligned}
		&V_t(t,x)+\inf_{u\in\mathbb R}
		\left\{b_tuV_x(t,x)+\frac12\sigma_t^2u^2V_{xx}(t,x)\right\}=0,
		&&(t,x)\in\mathcal D^\circ,\\[2mm]
		&V(t,\ell(t))=g_{\ell}(t),\qquad
		V(t,r(t))=g_{r}(t),
		&&0\le t\le T,\\[1mm]
		&V(T,x)=h(x),
		&&\ell(T)\le x\le r(T).
	\end{aligned}
	\right.
\end{equation}
Whenever $V_{xx}(t,x)>0$, the Hamiltonian is strictly convex in $u$, and its
unique minimizer is
\begin{equation}
	\label{eq:feedback}
	u^*(t,x)=-\frac{b_tV_x(t,x)}
	{\sigma_t^2V_{xx}(t,x)}.
\end{equation}
Substitution of \eqref{eq:feedback} into the HJB equation in
\eqref{eq:HJB-boundary-problem} yields
\begin{equation}
	\label{eq:reduced-HJB}
	V_t(t,x)-\kappa_t\frac{V_x(t,x)^2}{V_{xx}(t,x)}=0,\qquad
	\kappa_t:=\frac12\left(\frac{b_t}{\sigma_t}\right)^2.
\end{equation}
The condition $V_{xx}>0$ is essential in the present unconstrained problem.
Indeed, if $V_{xx}<0$, or if $V_{xx}=0$ and $V_x\neq0$, the infimum in the
HJB equation in \eqref{eq:HJB-boundary-problem} equals $-\infty$. Spatial
convexity is therefore part of the structure of the HJB equation and will be
established later rather than imposed on the value function from the outset.

\section{The candidate zero-volatility set and one-sided problems}
\label{sec:candidate-interface}

Throughout Sections~\ref{sec:candidate-interface}--\ref{sec:global}, we focus
on the configuration in which the candidate zero-volatility point remains in the
interior of the state domain:
\[
\ell(t)<0<r(t),\qquad 0\le t\le T.
\]
The one-sided configurations in which the entire state interval remains on
one side of zero, namely
\[
 0<\ell(t)<r(t)\quad\text{for all }t\in[0,T],
 \qquad\text{or}\qquad
 \ell(t)<r(t)<0\quad\text{for all }t\in[0,T],
\]
are treated in Appendix~\ref{app:one-sided-domain}.

Using $\kappa$ from \eqref{eq:reduced-HJB}, define
\begin{equation}\label{eq:time-change}
	s(t):=\int_t^T\kappa_\theta\,d\theta,
	\qquad S:=s(0).
\end{equation}
Assumption~\ref{ass:coefficients} makes $s$ strictly decreasing from
$[0,T]$ onto $[0,S]$.  Its $C^{1+\alpha}$ inverse is denoted by $t=t(s)$.
We first use this normalization to establish an estimate in the original
time variable.

\begin{Lemma}[Quadratic estimate at the origin]
	\label{lem:quadratic-origin}
	Under Assumption~\ref{ass:coefficients}, there exists a constant $M\ge1$,
	depending only on the prescribed data, such that
	\begin{equation}
		\label{eq:quadratic-bound}
		0\le V(t,x)
		\le
		M\exp\left(-2\int_t^T\kappa_\theta\,d\theta\right)x^2,
		\qquad (t,x)\in\mathcal D.
	\end{equation}
	Consequently, for every $t\in[0,T]$,
	\begin{equation}
		\label{eq:value-gradient-zero}
		V(t,0)=0,
		\qquad
		V_x(t,0)=0,
	\end{equation}
	where the spatial derivative at $x=0$ exists in the classical sense.
\end{Lemma}

\begin{proof}
The ratio $h(x)/x^2$ extends continuously at zero and is bounded on the
compact terminal interval.  The functions $\ell$ and $r$ stay uniformly away
from zero;
hence $M$ can be chosen so that
\[
 \overline V(t,x):=Me^{-2s(t)}x^2
\]
dominates both the lateral and terminal data.  

Fix $(t,x)\in\mathcal D^\circ$ and use, up to the exit time
$\widehat\tau$, the feedback
\[
\widehat u_\theta
=-\frac{b_\theta}{\sigma_\theta^2}X_\theta,
\qquad t\le\theta<\widehat\tau,
\]
with $\widehat u_\theta=0$ after $\widehat\tau$. Before exit, the
closed-loop equation is
\[
dX_\theta
=-\frac{b_\theta^2}{\sigma_\theta^2}X_\theta\,d\theta
-\frac{b_\theta}{\sigma_\theta}X_\theta\,dW_\theta.
\]
It has a unique strong solution, and the stopped feedback is admissible
because the coefficients and the state before exit are bounded.

Using $2\kappa_\theta=b_\theta^2/\sigma_\theta^2$, a direct calculation
gives
\[
\overline V_t
+b_\theta\widehat u_\theta\overline V_x
+\frac12\sigma_\theta^2\widehat u_\theta^2\overline V_{xx}
=0
\]
along the closed-loop state. The stochastic integrand in It\^o's formula is
bounded up to $\widehat\tau$; hence
$\overline V(\theta\wedge\widehat\tau,
X_{\theta\wedge\widehat\tau})$ is a martingale. The boundary and
terminal domination therefore imply
\[
V(t,x)
\le J(t,x;\widehat u)
\le
\mathbb E_{t,x}^{\widehat u}\left[
\overline V(\widehat\tau,X_{\widehat\tau})\right]
=
\overline V(t,x).
\]
The same bound holds on the parabolic boundary, while nonnegativity of the
payoffs gives $V\ge0$. This proves \eqref{eq:quadratic-bound}.
 Taking $x=0$ and then difference
quotients proves \eqref{eq:value-gradient-zero}.
\end{proof}

Lemma~\ref{lem:quadratic-origin} identifies the origin as a stationary point
of the value function at every time, but it does not by itself identify the
entire zero-volatility set. Strict convexity and the regularity needed to use the
feedback law \eqref{eq:feedback} pointwise will be established in
Sections~\ref{sec:auxiliary}--\ref{sec:global}. Once these properties are
available, the asymptotics at $x=0$ show that the optimal control, and hence
the optimal volatility, admits the natural continuous value zero at $x=0$.
Thus
\[
\Gamma_0:=\{(t,0):0\le t\le T\}
\]
is the natural candidate for the zero-volatility set.

Motivated by Lemma~\ref{lem:quadratic-origin}, we divide the original domain
along $\Gamma_0$. Define
\[
\mathcal D_+:=\{(t,x):0\le t\le T,\ 0\le x\le r(t)\},\quad
\mathcal D_-:=\{(t,x):0\le t\le T,\ \ell(t)\le x\le0\},
\]
with interiors
\[
\mathcal D_+^\circ:=\{(t,x):0\le t<T,\ 0<x<r(t)\},\quad
\mathcal D_-^\circ:=\{(t,x):0\le t<T,\ \ell(t)<x<0\}.
\]
We use these subdomains only to construct the two branches of a global
solution of \eqref{eq:HJB-boundary-problem}; no separate
one-sided value functions are introduced. On the left, reflect the state by
writing the positive coordinate as $x=-\xi$ when the original state is
$\xi<0$. With the backward time change \eqref{eq:time-change}, introduce the
two one-sided data sets
\begin{equation}\label{eq:side-data}
	\begin{aligned}
		R_+(s)&:=r(t(s)),
		& Q_+(s)&:=g_r(t(s)),
		& h_+(x)&:=h(x),\\
		R_-(s)&:=-\ell(t(s)),
		& Q_-(s)&:=g_\ell(t(s)),
		& h_-(x)&:=h(-x).
	\end{aligned}
\end{equation}
Here $h_+$ is considered on $[0,R_+(0)]$, whereas $h_-$ is considered
on $[0,R_-(0)]$.

Formally, a sufficiently smooth strictly convex solution of
\eqref{eq:HJB-boundary-problem}, restricted to either side of
zero and expressed in the positive one-sided coordinate, leads to the
following pair of auxiliary problems:
\begin{equation}
	\label{eq:right-transformed-problem}
	\begin{cases}
		v_s^\pm(s,x)
		+\dfrac{v_x^\pm(s,x)^2}{v_{xx}^\pm(s,x)}=0,
		&0<x<R_\pm(s),\quad 0<s\le S,\\[3mm]
		v^\pm(0,x)=h_\pm(x),
		&0\le x\le R_\pm(0),\\[2mm]
		v^\pm(s,0)=0,\qquad
		v^\pm\bigl(s,R_\pm(s)\bigr)=Q_\pm(s),
		&0\le s\le S.
	\end{cases}
\end{equation}

The strict convexity used in these formal reductions is not assumed at this
stage. The problems in \eqref{eq:right-transformed-problem} only motivate
the independent dual construction developed in the next section. Since the
two one-sided problems have the same structure, we carry out the dual
construction for the right-hand data and subsequently apply the same argument
to the reflected left-hand data. Section~\ref{sec:global} will first
reconstruct and paste the two primal branches into a classical solution of
\eqref{eq:HJB-boundary-problem}, and will then identify that solution with the
value function \eqref{eq:value} by a separate verification argument.

\section{Dual free-boundary problem}
\label{sec:auxiliary}

\subsection{Formal duality and boundary comparison}
\label{subsec:formal-dual}

Since, after reflection, the left-hand problem differs from the right-hand
problem only through the data $(R_-,\allowbreak Q_-,\allowbreak h_-)$, the dual construction and the
arguments below apply to both problems in exactly the same way. It therefore
suffices to consider the right-hand problem. Throughout the remainder of this
section, we consider only the right-hand dual problem; the corresponding
results for the reflected left-hand problem follow by replacing
$(R_+,Q_+,h_+)$ with $(R_-,Q_-,h_-)$.

Suppose formally that $v^+(s,\cdot)$ is strictly convex and write its
restricted conjugate as
\[ 
w(s,p):=\sup_{0\le x\le R_+(s)}\{px-v^+(s,x)\}. 
\] 
At an interior dual pair, $p=v_x^+(s,x)$ and 
\[ 
x=w_p(s,p),\qquad 
v_{xx}^+(s,x)=\frac1{w_{pp}(s,p)}, 
\qquad 
v_s^+(s,x)=-w_s(s,p). 
\] 
Hence \eqref{eq:right-transformed-problem} formally becomes the free-boundary
system
\begin{equation}\label{eq:w-free-boundary-problem}
	\begin{gathered}
		w_s-p^2w_{pp}=0,
		\qquad 0<p<p_R(s),\quad 0<s\le S,\\
		w(s,0)=0,\qquad w(0,p)=H_+(p),\\
		w_p(s,p_R(s))=R_+(s),\qquad
		p_R(s)R_+(s)-w(s,p_R(s))=Q_+(s).
	\end{gathered}
\end{equation} 

Here the restricted terminal conjugate is 
\begin{equation}\label{eq:H-obstacle} 
	H_+(p):=\sup_{0\le x\le R_+(0)}\{px-h_+(x)\}, 
	\qquad p\ge0. 
\end{equation} 

A one-sided extension $\widehat h_+$ of $h_+$ is called admissible if it 
is superlinear, agrees with $h_+$ on $[0,R_+(0)]$, and satisfies 
\[ 
\widehat h_+\in C^2([0,\infty))\cap C^3((0,\infty)), 
\qquad 
\widehat h_+''>0\quad\text{on }(0,\infty). 
\] 
Its full conjugate is 
\begin{equation}\label{eq:full-conjugates} 
	H_{\infty,+}(p):=\sup_{x\ge0}\{px-\widehat h_+(x)\}, 
	\qquad p\ge0. 
\end{equation} 
For $p>0$, standard conjugacy gives 
\[ 
H_{\infty,+}'(p)=(\widehat h_+')^{-1}(p), 
\qquad 
H_{\infty,+}''(p) 
=\frac1{\widehat h_+''((\widehat h_+')^{-1}(p))}>0, 
\] 
and hence $H_{\infty,+}\in C^3((0,\infty))$.  We also require 
$\partial_p^kH_{\infty,+}$ to have at most polynomial growth as 
$p\to\infty$ for $k=0,1,2,3$. 

Let
\begin{equation}\label{eq:GM}
	Z_\theta:=e^{\sqrt2B_\theta-\theta},
	\qquad
	P_\theta^p:=pZ_\theta,
	\qquad \theta\ge0,\quad p\ge0.
\end{equation}
where $B$ is a standard Brownian motion.
For an admissible extension $\widehat h_+$, define
\begin{equation}\label{eq:free-conjugate-semigroup}
	\mathscr H_+(s,p):=\mathbb E[H_{\infty,+}(P_s^p)],
	\qquad s\ge0,\quad p\ge0.
\end{equation}
Lognormal moments justify
differentiation under the expectation and give 
$(\mathscr H_+)_{pp}>0$ for $s>0$ and $p>0$. 
Strict convexity gives a unique function $\overline p_+(s)>0$ satisfying 
\begin{equation}\label{eq:comparison-contact} 
	(\mathscr H_+)_p(s,\overline p_+(s))=R_+(s).
\end{equation} 
Consequently, the restriction of $\mathscr H_+$ to
$0\le p\le\overline p_+(s)$ satisfies
\begin{equation}\label{eq:comparison-domain}
	(\mathscr H_+)_s-p^2(\mathscr H_+)_{pp}=0,
	\qquad 0<p<\overline p_+(s),\quad 0<s\le S,
\end{equation}
with
\begin{equation}\label{eq:comparison-boundary}
	\begin{gathered}
		\mathscr H_+(s,0)=0,
		\qquad \mathscr H_+(0,p)=H_{\infty,+}(p),\\
		(\mathscr H_+)_p(s,\overline p_+(s))=R_+(s).
	\end{gathered}
\end{equation}
Thus $\mathscr H_+$ provides an explicit comparison solution whose slope
matches the moving state boundary at $\overline p_+$.

The affine payoff associated with the moving state boundary is
\begin{equation}\label{eq:Phi-obstacle}
	\Phi_+(s,p):=pR_+(s)-Q_+(s).
\end{equation}

\begin{Assumption}[Boundary comparison]
	\label{ass:dual-comparison}
	There exist admissible extensions $\widehat h_+$ and $\widehat h_-$ such
	that, for each sign, $0\le s\le S$, and
	$p\ge\overline p_\pm(s)$,
	\begin{equation}\label{eq:dual-comparison-condition}
		pR_\pm'(s)-Q_\pm'(s)
		\ge
		(\mathscr H_\pm)_s
		\bigl(s,\overline p_\pm(s)\bigr).
	\end{equation}
\end{Assumption}

Since \eqref{eq:dual-comparison-condition} holds for every
$p\ge\overline p_\pm(s)$, it also implies $R_\pm'(s)\ge0$.

\begin{Remark}
	\label{rem:boundary-comparison}
	For the right-hand problem, the classical compatibility condition for a
	regular parabolic free-boundary
	problem requires
	\[
	(\partial_s-p^2\partial_{pp})\Phi_+
	=pR_+'-Q_+'\ne0;
	\]
	see \cite[conditions~(1.7)$'$ and~(1.14)]{FasanoPrimicerio1979}
	and the related formulations in \cite{Sherman1971,Meyer1977}.
	In the present convex setting, the relevant sign is positive.
	Assumption~\ref{ass:dual-comparison} strengthens this condition to
	\[
	pR_+'-Q_+'
	\ge (\mathscr H_+)_s(s,\overline p_+)>0,
	\qquad p\ge\overline p_+(s).
	\]
	The additional margin yields an upper bound for the free
	boundary.  The left-hand condition follows by reflection.
\end{Remark}

The assumption is nonempty.  For example, take
$R_\pm\equiv\rho_\pm>0$ and $h_\pm(x)=x^2/2$.  Then
\[
\mathscr H_\pm(s,p)=\frac12e^{2s}p^2,
\qquad
(\mathscr H_\pm)_s(s,\overline p_\pm(s))
=\rho_\pm^2e^{-2s}.
\]
Thus
\[
Q_\pm(s)=\frac12\rho_\pm^2e^{-2s}-\delta_\pm s,
\qquad \delta_\pm\ge0,
\]
satisfies \eqref{eq:dual-comparison-condition} whenever it remains positive.

Fix these extensions henceforth.  Although $\mathscr H_+$,
$\overline p_+$ depend on the chosen extension, $H_+$ and the
stopping problem below do not.  Since neither the required strict convexity
nor the existence of $p_R$ has yet been proved, we do not solve
\eqref{eq:w-free-boundary-problem} directly. Instead, the
next subsection constructs the dual solution independently as an
optimal-stopping value and recovers $p_R$ as its stopping boundary.

\subsection{Optimal stopping construction of the dual free boundary}
\label{subsec:dual-obstacle}

The restricted conjugate $H_+$ in \eqref{eq:H-obstacle} satisfies
$H_+\ge\Phi_+(0,\cdot)$, with equality if and only if
$p\ge h_+'(R_+(0))$.

For the process $P^p$ introduced in
\eqref{eq:GM},
\[
dP_\theta^p=\sqrt{2}P_\theta^p\,dB_\theta.
\]
Thus $P^p$ is a nonnegative martingale with generator
$p^2\partial_{pp}$. For $(s,p)\in[0,S]\times[0,\infty)$, define
the optimal stopping value
\begin{equation}
	\label{eq:dual-obstacle-stopping}
	\mathcal{W}(s,p)
	:=
	\sup_{0\le\tau\le s}
	\mathbb E\left[
	\Phi_+(s-\tau,P_\tau^p)\mathbf 1_{\{\tau<s\}}
	+
	H_+(P_s^p)\mathbf 1_{\{\tau=s\}}
	\right],
\end{equation}
where the supremum is taken over all stopping times with values in $[0,s]$.
Since both rewards have at most linear growth, the martingale property of
$P^p$ ensures that $\mathcal W$ is finite.

Define the stopping and continuation regions by
\[
\mathcal S:=\{(s,p):\mathcal{W}(s,p)=\Phi_+(s,p)\},
\qquad
\mathcal C:=\{(s,p):\mathcal{W}(s,p)>\Phi_+(s,p)\}.
\]
For each $s\in[0,S]$, denote their spatial sections by
\[
\mathcal S_s:=\{p\ge0:(s,p)\in\mathcal S\},
\qquad
\mathcal C_s:=\{p\ge0:(s,p)\in\mathcal C\}.
\]

\begin{Proposition}
	\label{prop:stopping-boundary}
	Under Assumptions~\ref{ass:coefficients}
	and~\ref{ass:dual-comparison},
	$\mathcal W$ is continuous on $[0,S]\times[0,\infty)$. For every
	$s\in[0,S]$, the function $\mathcal W(s,\cdot)$ is nondecreasing and
	convex, and
	\begin{equation}
		\label{eq:W-growth}
		0\le\mathcal W(s,p)\le R_+(s)p,
		\qquad p\ge0.
	\end{equation}

	There exists a unique stopping boundary
	$\beta:[0,S]\to(0,\infty)$ such that
	\begin{equation*}
		\mathcal S_s=[\beta(s),\infty),
		\qquad
		\mathcal C_s=[0,\beta(s)).
	\end{equation*}
	The function $\beta$ is lower semicontinuous on $[0,S]$ and satisfies
	\begin{equation}
		\label{eq:beta-bounds}
		\frac{Q_+(s)}{R_+(s)}
		\le\beta(s)\le\overline p_+(s),
		\qquad 0\le s\le S,
	\end{equation}
	and
	\begin{equation}
		\label{eq:beta-initial-limit}
		\beta(0)=h_+'(R_+(0)),
		\qquad
		\lim_{s\downarrow0}\beta(s)=h_+'(R_+(0)).
	\end{equation}

	Finally, $(\mathcal W,\beta)$ satisfies
	\eqref{eq:w-free-boundary-problem} with
	$(w,p_R)$ replaced by $(\mathcal W,\beta)$, where the derivative at
	$p=\beta(s)$ is understood from the continuation region.
\end{Proposition}

The proof is given in Appendix~\ref{app:one-obstacle-boundary}.

The bounds \eqref{eq:beta-bounds}, positivity of the costs, and compactness
give constants $c_*,c^*>0$ such that
\begin{equation}\label{eq:beta-uniform-bounds}
0<c_*\le \beta(s)\le c^*<\infty,
\qquad 0\le s\le S.
\end{equation}

\subsection{Asymptotics near \texorpdfstring{$p=0$}{p = 0}}
\label{sec:asymptotics}

By \eqref{eq:beta-uniform-bounds}, a fixed neighborhood of $p=0$ lies in the continuation region for
every $s$.

\begin{Proposition}[Superalgebraic decay of the boundary correction]
	\label{prop:flat-correction}
	Under Assumptions~\ref{ass:coefficients}
	and~\ref{ass:dual-comparison}, choose $p_0$ so that
\[
0<p_0<\min\left\{\inf_{0\le s\le S}\beta(s),\ h_+'(R_+(0))\right\}.
\]
Then, for $k=0,1,2,3$ and every $N>0$,
\begin{equation}\label{eq:flat-correction}
 \sup_{0\le s\le S}
 \left|\partial_p^k\bigl(\mathcal W(s,p)-\mathscr H_+(s,p)\bigr)\right|
 =O(p^N),\qquad p\downarrow0.
\end{equation}
\end{Proposition}

\begin{proof}
Set
\[
 e(s,p):=\mathcal W(s,p)-\mathscr H_+(s,p).
\]
Because $p_0<\inf_{0\le s\le S}\beta(s)$, the fixed strip
$(0,S]\times(0,p_0)$ lies in the continuation region. Hence both
$\mathcal W$ and $\mathscr H_+$ are classical there and satisfy
$u_s-p^2u_{pp}=0$.

We first identify all the data for their difference. At $s=0$,
\[
 \mathcal W(0,p)=H_+(p),\qquad
 \mathscr H_+(0,p)=H_{\infty,+}(p).
\]
For $0<p\le p_0<h_+'(R_+(0))$, the maximizer in the full conjugate is
\[
 x_p=(\widehat h_+')^{-1}(p)<R_+(0).
\]
It therefore also maximizes the restricted conjugate, so
$H_+(p)=H_{\infty,+}(p)$ on $[0,p_0]$. Thus $e(0,p)=0$ there. Moreover,
\eqref{eq:W-growth} gives $\mathcal W(s,0)=0$, while
$\mathscr H_+(s,0)=H_{\infty,+}(0)=0$. If
\[
 \psi(s):=e(s,p_0),
\]
then continuity gives $\psi\in C([0,S])$ and $\psi(0)=0$. Consequently,
$e$ is a bounded solution on $[0,S]\times[0,p_0]$ of
\begin{equation}\label{eq:flat-error-equation}
 \begin{cases}
  e_s-p^2e_{pp}=0,
  &0<s\le S,\quad 0<p<p_0,\\
  e(0,p)=0,\qquad e(s,0)=0,\\
  e(s,p_0)=\psi(s).
 \end{cases}
\end{equation}
Here $e(0,p_0)=\psi(0)=0$ is the corner compatibility condition. Since
$e$ is continuous on the compact rectangle, the last condition also implies
\begin{equation}\label{eq:flat-uniform-endpoint-trace}
 \lim_{p\downarrow0}\sup_{0\le s\le S}|e(s,p)|=0.
\end{equation}

Now set
\[
 z:=\log\frac{p_0}{p},
 \qquad E(s,z):=e(s,p_0e^{-z}).
\]
The chain rule gives $p^2e_{pp}=E_{zz}+E_z$. Hence
\[
 E_s-E_{zz}-E_z=0,\qquad
 E(0,z)=0,\qquad E(s,0)=\psi(s),
\]
and \eqref{eq:flat-uniform-endpoint-trace} becomes
$E(s,z)\to0$ as $z\to\infty$, uniformly in $s$. After the gauge transform
$U=e^{z/2+s/4}E$, the function $U$ satisfies the standard heat equation on
the half-line. The Dirichlet Poisson formula
\cite[Chapter~6]{Cannon1984} therefore yields
\begin{equation}\label{eq:flat-drifted-poisson}
 E(s,z)=\int_0^s
 \frac{z}{2\sqrt\pi r^{3/2}}
 \exp\!\left(-\frac{z^2}{4r}-\frac z2-\frac r4\right)
 \psi(s-r)\,dr.
\end{equation}
Indeed, for $z>0$ differentiation under the integral verifies the equation;
the heat Poisson kernel is an approximation to the identity at $z=0$, giving
$E(s,0)=\psi(s)$, while the zero initial condition and the limit at infinity
follow directly from the formula. Uniqueness for bounded solutions, by the
maximum principle, identifies \eqref{eq:flat-drifted-poisson} with $E$.

It remains only to estimate the explicit kernel. For $j=0,1,2,3$, splitting
$e^{-z^2/(4r)}$ into two Gaussian factors and using $0<r\le S$ gives
\begin{equation}\label{eq:flat-log-derivatives}
 \sup_{0\le s\le S}|\partial_z^jE(s,z)|
 \le C_j e^{-cz^2},
 \qquad z\ge1.
\end{equation}
The negative powers of $r$ arising after differentiation are absorbed by the
second Gaussian factor, as is seen from the substitution $q=z^2/(8r)$.

Finally, $\partial_p=-p^{-1}\partial_z$, and for $k=1,2,3$,
\[
 \partial_p^k e(s,p)
 =(-1)^kp^{-k}\prod_{j=0}^{k-1}(\partial_z+j)E(s,z),
 \qquad z=\log\frac{p_0}{p}.
\]
Thus \eqref{eq:flat-log-derivatives} implies
\[
 \sup_{0\le s\le S}|\partial_p^ke(s,p)|
 \le C_kp^{-k}
 \exp\!\left[-c\left(\log\frac{p_0}{p}\right)^2\right]
 =o(p^N)
\]
for every $N>0$ and $k=0,1,2,3$. This proves
\eqref{eq:flat-correction}.
\end{proof}

The stopping construction and Proposition~\ref{prop:flat-correction} require
no power-law behavior of the terminal loss at the origin.  To identify the
leading behavior near $p=0$ and the scale of the optimal feedback, we now impose
the following additional local condition.

\begin{Assumption}[Smooth finite-order terminal minimum]
	\label{ass:finite-order}
	For each sign, there exist $m_\pm\ge2$ and $c_\pm>0$ such that
	\begin{equation}\label{eq:h-finite-order}
		\partial_x^k h_\pm(x)
		=
		\partial_x^k\left(\frac{c_\pm}{m_\pm}x^{m_\pm}\right)
		+o\bigl(x^{m_\pm-k}\bigr),
		\qquad k=0,1,2,3,
		\qquad x\downarrow0.
	\end{equation}
\end{Assumption}

The two sides may have different orders.  Define
\begin{equation}\label{eq:q-and-Ch}
	q_\pm:=\frac{m_\pm}{m_\pm-1},
	\qquad
	C_{h,\pm}:=\frac1{q_\pm}c_\pm^{-1/(m_\pm-1)}.
\end{equation}
For sufficiently small $p>0$, the maximizer in the definition of
$H_{\infty,\pm}(p)$ lies near the origin, where the admissible extension
agrees with $h_\pm$.  Conjugate inversion therefore gives
\[
(H_{\infty,\pm})_p(p)=(h_\pm')^{-1}(p),\qquad
(H_{\infty,\pm})_{pp}(p)
=\frac{1}{h_\pm''((h_\pm')^{-1}(p))},
\]
and
\[
(H_{\infty,\pm})_{ppp}(p)
=
-\frac{h_\pm'''((h_\pm')^{-1}(p))}
{\bigl[h_\pm''((h_\pm')^{-1}(p))\bigr]^3}.
\]
Substituting \eqref{eq:h-finite-order} into these identities yields
\begin{equation}\label{eq:smooth-finite-order}
	\partial_p^kH_{\infty,\pm}(p)
	=
	\partial_p^k\bigl(C_{h,\pm}p^{q_\pm}\bigr)
	+o\bigl(p^{q_\pm-k}\bigr),
	\qquad k=0,1,2,3,
	\qquad p\downarrow0.
\end{equation}
If $h''(0)>0$, the $C^2$ regularity of the terminal loss and
\eqref{eq:h-finite-order} necessarily give
\[
m_+=m_-=2,
\qquad
c_+=c_-=h''(0).
\]
Convexity and $C^2$ regularity alone do not imply
Assumption~\ref{ass:finite-order}, particularly when $h''(0)=0$, because
they provide no corresponding control of the third-order remainder.

\begin{Proposition}[Dual asymptotics near $p=0$]
	\label{prop:asymptotics_p0}
	Suppose that Assumptions~\ref{ass:coefficients},
	\ref{ass:dual-comparison}, and~\ref{ass:finite-order} hold.
	Then, uniformly for $s\in[0,S]$ as $p\downarrow0$,
	\begin{equation}\label{eq:W-general-power}
		\partial_p^k \mathcal W(s,p)
		=
		\partial_p^k\!\left(
		C_{h,+}e^{q_+(q_+-1)s}p^{q_+}
		\right)
		+o(p^{q_+-k}),
		\qquad k=0,1,2,3.
	\end{equation}
	In particular, there exists $p_*>0$ such that
	\begin{equation}\label{eq:dual-second-derivative-seed}
		\mathcal W_{pp}(s,p)>0,
		\qquad
		0\le s\le S,\quad 0<p\le p_*,
	\end{equation}
	and
	\begin{equation}\label{eq:weighted-dual-second-derivative-near-zero}
		\sup_{\substack{0\le s\le S\\0<p\le p_*}}
		p\,\mathcal W_{pp}(s,p)<\infty.
	\end{equation}
\end{Proposition}

\begin{proof}
For $k=0,1,2,3$, differentiation under the expectation is justified by the
polynomial-growth assumption on $H_{\infty,+}^{(k)}$ and the fact that lognormal
random variables have moments of every order, uniformly for $s\in[0,S]$.
Thus
\begin{equation}\label{eq:semigroup-derivatives}
 \partial_p^k\mathscr H_+(s,p)
 =\mathbb E\!\left[Z_s^kH_{\infty,+}^{(k)}(pZ_s)\right].
\end{equation}
Write
\[
 A_{k,+}:=C_{h,+}q_+(q_+-1)\cdots(q_+-k+1),
 \qquad A_{0,+}:=C_{h,+},
\]
where the empty product is one.  For $k=3$ and $q_+=2$, this convention gives
$A_{3,+}=0$, consistently with \eqref{eq:smooth-finite-order}. We claim that,
uniformly in $s$,
\begin{equation}\label{eq:free-semigroup-asymptotics}
 p^{k-q_+}\partial_p^k\mathscr H_+(s,p)
 \longrightarrow A_{k,+}\mathbb E Z_s^{q_+}
 =A_{k,+}e^{q_+(q_+-1)s}.
\end{equation}
Fix $\delta>0$ so that $r^{k-q_+}H_{\infty,+}^{(k)}(r)$ is bounded on
$(0,\delta]$ and split \eqref{eq:semigroup-derivatives} according to
$\{pZ_s\le\delta\}$ and its complement. On the first event the normalized
integrand is
\[
 Z_s^{q_+}(pZ_s)^{k-q_+}H_{\infty,+}^{(k)}(pZ_s),
\]
and dominated convergence applies. To see uniformity, first restrict to
$Z_s\le K$, where convergence at zero is uniform, and then let $K\to\infty$
using uniform lognormal moment bounds. On $\{pZ_s>\delta\}$, polynomial
growth and Markov's inequality with an arbitrarily high lognormal moment show
that the normalized expectation is $O(p^M)$ for any prescribed $M>0$,
uniformly in $s$. This proves \eqref{eq:free-semigroup-asymptotics}.
Proposition~\ref{prop:flat-correction} now transfers the four estimates from
$\mathscr H_+$ to $\mathcal W$ and proves \eqref{eq:W-general-power}.

Since
\[
C_{h,+}q_+(q_+-1)e^{q_+(q_+-1)s}>0,
\]
the case $k=2$ of \eqref{eq:W-general-power} gives
\eqref{eq:dual-second-derivative-seed}.  Moreover,
\[
p\,\mathcal W_{pp}(s,p)
=
C_{h,+}q_+(q_+-1)e^{q_+(q_+-1)s}p^{q_+-1}
+o(p^{q_+-1}),
\]
uniformly in $s$.  Since $q_+>1$, this proves
\eqref{eq:weighted-dual-second-derivative-near-zero}.
\end{proof}

The same argument applies to the left-side dual problem with $+$ replaced
by $-$.

\subsection{Strict convexity}
\label{subsec:strict-convexity}

Convexity follows from the stopping construction. Assumption~\ref{ass:finite-order}
and Proposition~\ref{prop:asymptotics_p0} provide strict positivity of
$\mathcal W_{pp}$ in a fixed neighborhood of $p=0$, uniformly in time.  The strong
minimum principle then extends strict positivity throughout the continuation
region.  No second-order trace at the free boundary is required.

\begin{Corollary}[Strict convexity]
	\label{cor:strict-convexity}
	Under the hypotheses of
	Proposition~\ref{prop:asymptotics_p0}, for every $s\in[0,S]$,
	\begin{equation}
		\label{eq:strict-convexity-interior}
		\mathcal W_{pp}(s,p)>0,
		\qquad 0<p<\beta(s).
	\end{equation}
\end{Corollary}

\begin{proof}
	At $s=0$, we have $\mathcal W(0,\cdot)=H_+$ and $H_+''>0$ on
	$(0,h_+'(R_+(0)))$. For $s>0$, the convexity established in
	Proposition~\ref{prop:stopping-boundary}, together with interior regularity,
	gives $\mathcal W_{pp}\ge0$ in $\mathcal C$. Set
	\[
	\zeta(s,p):=e^{-2s}\mathcal W_{pp}(s,p).
	\]
	Differentiating $\mathcal W_s-p^2\mathcal W_{pp}=0$ twice with respect to $p$ gives
	\begin{equation}
		\label{eq:zeta-equation}
		\zeta_s-p^2\zeta_{pp}-4p\zeta_p=0
		\qquad\text{in }\mathcal C.
	\end{equation}

	Proposition~\ref{prop:asymptotics_p0} shows that $\zeta>0$ in a fixed
	neighborhood of $p=0$, uniformly in $s$. If $\zeta(s_*,p_*)=0$ at some
	$(s_*,p_*)\in\mathcal C$ with $s_*>0$, openness and the threshold structure
	of $\mathcal C$ provide $\delta,\rho>0$ and a sufficiently small
	$\varepsilon>0$ such that
	\[
	[s_*-\delta,s_*]\times[\varepsilon,p_*+\rho]\subset\mathcal C
	\quad\text{and}\quad
	\zeta(s,\varepsilon)>0
	\quad(s_* -\delta\le s\le s_*).
	\]
	Equation~\eqref{eq:zeta-equation} is uniformly parabolic on this cylinder,
	so the strong minimum principle contradicts $\zeta(s_*,p_*)=0$. Thus
	$\zeta>0$ throughout $\mathcal C$, proving
	\eqref{eq:strict-convexity-interior}. The same one-sided argument applies at
	$s_*=S$.
\end{proof}

\begin{Lemma}[Uniform bound for $p\mathcal W_{pp}$]
	\label{lem:uniform-weighted-dual-second-derivative}
	Suppose that Assumptions~\ref{ass:coefficients},
	\ref{ass:dual-comparison}, and~\ref{ass:finite-order} hold. Then
	\begin{equation}
		\label{eq:uniform-weighted-dual-second-derivative}
		\sup_{\substack{0\le s\le S\\0<p<\beta(s)}}
		p\,\mathcal W_{pp}(s,p)<\infty.
	\end{equation}
	The same conclusion holds for the dual problem obtained from the reflected
	left-hand data.
\end{Lemma}

The proof is given in
Appendix~\ref{app:uniform-weighted-dual-second-derivative}.

\section{Global classical solution and optimality}
\label{sec:global}

The strict convexity obtained in
Corollary~\ref{cor:strict-convexity} allows us to invert the dual
transformation. We first reconstruct two auxiliary branches and paste them at
$x=0$ to obtain a global classical solution of
\eqref{eq:HJB-boundary-problem}. We then verify separately that
this solution is the value function \eqref{eq:value}.

\subsection{Construction of a global classical solution}
\label{subsec:regional-reconstruction}

Recall the time change $s=s(t)$ introduced in
\eqref{eq:time-change}. Since $\kappa_t>0$, it maps $[0,T]$
bijectively onto $[0,S]$ with reversed orientation. Throughout this section,
$t=t(s)$ denotes its inverse.

For $0\le s\le S$ and $0\le x\le r(t(s))$, define
\begin{equation}
	\label{eq:vhat-plus}
	\widehat v^+(s,x)
	:=
	\sup_{0\le p\le \beta(s)}\{px-\mathcal W(s,p)\}.
\end{equation}

Let $\widehat v^-$ denote the analogous inverse conjugate for the reflected
problem with positive state variable $y=-x$, upper boundary $-\ell(t)$, and
boundary cost $g_\ell(t)$.

Returning directly to the original time and state variables, define the global
candidate by
\begin{equation}
	\label{eq:global-candidate}
	\widehat v(t,x):=
	\begin{cases}
		\widehat v^-\bigl(s(t),-x\bigr),&\ell(t)\le x<0,\\
		0,&x=0,\\
		\widehat v^+\bigl(s(t),x\bigr),&0<x\le r(t).
	\end{cases}
\end{equation}
Here $\widehat v^\pm$ are expressed in the backward time variable $s$, whereas
the global function $\widehat v$ is expressed in the original variables
$(t,x)$.

In the following statement, $0^+$ and $0^-$ denote the one-sided traces
from the right and left, respectively.

\begin{Theorem}[Construction of a global classical solution]
	\label{thm:regional-reconstruction}
	Suppose that Assumptions~\ref{ass:coefficients},
	\ref{ass:dual-comparison}, and~\ref{ass:finite-order} hold.
	Then the function $\widehat v$ in \eqref{eq:global-candidate} satisfies
	\[
		\widehat v\in C^{1,2}(\mathcal D^\circ)\cap C(\mathcal D),
		\qquad
		\widehat v_{xx}>0\quad\text{when }x\ne0,
	\]
	and solves \eqref{eq:HJB-boundary-problem} classically. At
	$x=0$, for $0\le t<T$,
	\begin{equation}\label{eq:global-interface-traces}
		\widehat v(t,0)=\widehat v_x(t,0)=\widehat v_t(t,0)=0,
		\qquad
		\widehat v_{xx}(t,0)=
		\begin{cases}
			h''(0)e^{-2s(t)},&h''(0)>0,\\
			0,&h''(0)=0.
		\end{cases}
	\end{equation}

Moreover, for $k=0,1,2$,
\begin{equation}\label{eq:primal-endpoint-asymptotics}
\partial_x^k\widehat v(t,x)
=\partial_x^k\left[
\frac{c_\pm}{m_\pm}e^{-q_\pm s(t)}|x|^{m_\pm}
\right]+o(|x|^{m_\pm-k}),
\qquad x\to0^\pm,
\end{equation}
locally uniformly for $t\in[0,T)$.

	The pointwise Hamiltonian minimizer, continuously extended at $x=0$, is
	\begin{equation}
		\label{eq:global-feedback}
		\widehat u(t,x)
		:=
		\begin{cases}
		-\dfrac{b_t\widehat v_x(t,x)}
		{\sigma_t^2\widehat v_{xx}(t,x)},&x\ne0,\\[3mm]
		0,&x=0
		\end{cases}
		\qquad (t,x)\in\mathcal D^\circ.
	\end{equation}
	It is bounded and locally Lipschitz in $x$ on
	$\mathcal D^\circ$, and
	\begin{equation}
		\label{eq:global-feedback-interface-regularity}
		\widehat u(t,x)
		=-\frac{b_t}{\sigma_t^2(m_\pm-1)}x+o(|x|),
		\qquad x\to0^\pm,
	\end{equation}
	where the remainder is locally uniform for $t\in[0,T)$.
\end{Theorem}
\begin{proof}
We give the proof for the right one-sided problem.  The left-hand result
	follows by applying the same argument to $R_-$, $Q_-$, and $h_-$.
	The exponents $m_\pm$ and coefficients $c_\pm$ come from
	Assumption~\ref{ass:finite-order}. In particular, on the right-hand side,
	$m_+\ge2$ is the order at which the terminal loss leaves its minimum at
	zero, in the sense that
	\[
		h_+'(x)\sim c_+x^{m_+-1}\qquad(x\downarrow0),
	\]
	by \eqref{eq:h-finite-order}, while $q_+=m_+/(m_+-1)$ is defined in
	\eqref{eq:q-and-Ch}.
	
	Let $0<c_*\le \beta(s)\le c^*$ be the bounds in
	\eqref{eq:beta-uniform-bounds}.  Since
	$\mathcal{W}(s,p)=pR_+(s)-Q_+(s)$ for $p\ge \beta(s)$, one has, for
	$0\le x\le R_+(s)$,
	\begin{equation}\label{eq:vhat-fixed-dual-domain}
		\widehat v^+(s,x)
		=\sup_{0\le p\le c^*}\{px-\mathcal{W}(s,p)\}.
	\end{equation}
	Indeed, for $p\ge \beta(s)$ the expression in braces equals
	$p(x-R_+(s))+Q_+(s)$ and is nonincreasing in $p$.  The fixed-domain
	representation and the continuity of $\mathcal{W}$ show that
$\widehat v^+$ is continuous on the closed one-sided domain.
	
	By Corollary~\ref{cor:strict-convexity}, smooth fit, and
	Proposition~\ref{prop:asymptotics_p0},
	\[
	\mathcal W_p(s,0)=0,\qquad
	\mathcal W_p(s,\beta(s)^-)=R_+(s),\qquad
	\mathcal W_{pp}(s,p)>0\quad(0<p<\beta(s)).
	\]
	Thus, for $0<x<R_+(s)$, the supremum in
	\eqref{eq:vhat-fixed-dual-domain} is attained at the unique
	$p=p(s,x)\in(0,\beta(s))$ satisfying
	\begin{equation}\label{eq:inverse-dual-relation}
		\mathcal W_p(s,p(s,x))=x.
	\end{equation}
	The implicit function theorem and the dual equation give
	\begin{equation}\label{eq:inverse-legendre-identities}
		\widehat v_x^+=p,\qquad
		\widehat v_{xx}^+=\frac1{\mathcal W_{pp}},\qquad
		\widehat v_s^+=-\mathcal W_s=-p^2\mathcal W_{pp},
	\end{equation}
	where all dual quantities are evaluated at $(s,p(s,x))$.  Consequently,
	\[
	\widehat v_s^+
	+\frac{(\widehat v_x^+)^2}{\widehat v_{xx}^+}=0.
	\]
	By \eqref{eq:global-candidate},
	$\widehat v(t,x)=\widehat v^+(s(t),x)$ for $x>0$. Since
	$ds/dt=-\kappa_t$, the preceding equation is equivalent to the restriction
	of the HJB equation in \eqref{eq:HJB-boundary-problem} to
	$\mathcal D_+^\circ$, and $\widehat v_{xx}>0$ there.
	
We now derive the behavior near $x=0$ after returning to the state variable.
Set
	\[
	K_+(s):=q_+C_{h,+}e^{q_+(q_+-1)s}
	=c_+^{-1/(m_+-1)}e^{q_+(q_+-1)s}.
	\]
	The case $k=1$ of \eqref{eq:W-general-power} gives, uniformly in
	$s\in[0,S]$,
	\[
	x=\mathcal W_p(s,p)=K_+(s)p^{q_+-1}(1+o(1)).
	\]
	Since $K_+$ is continuous and bounded away from zero, strict monotonicity
	permits uniform inversion:
	\begin{equation}\label{eq:inverse-endpoint-p}
		p(s,x)=c_+e^{-q_+s}x^{m_+-1}(1+o(1)),
		\qquad x\downarrow0.
	\end{equation}
	Moreover, \eqref{eq:W-general-power} yields
	\[
	\frac{p\mathcal W_{pp}(s,p)}{\mathcal W_p(s,p)}
	\longrightarrow q_+-1=\frac1{m_+-1},
	\qquad
	\frac{\mathcal W(s,p)}{p\mathcal W_p(s,p)}
	\longrightarrow\frac1{q_+}.
	\]
	Recall from the original definition \eqref{eq:vhat-plus} that
	\[
		\widehat v^+(s,x)
		=\sup_{0\le p\le\beta(s)}\{px-\mathcal W(s,p)\}.
	\]
	At the interior maximizer this becomes
	$\widehat v^+=px-\mathcal W$, with $x=\mathcal W_p(s,p)$. Hence
	\eqref{eq:inverse-legendre-identities} gives
\eqref{eq:primal-endpoint-asymptotics} after the time change $s=s(t)$.
It also yields, uniformly for $s\in[0,S]$,
\[
\frac{\widehat v_x^+(s,x)}{\widehat v_{xx}^+(s,x)}
\sim\frac{x}{m_+-1},\qquad x\downarrow0.
\]
	For $0<s\le S$, the last identity in
	\eqref{eq:inverse-legendre-identities} also gives
	\[
	\widehat v_s^+(s,x)
	\sim-\frac{c_+}{m_+-1}e^{-q_+s}x^{m_+}.
	\]
	
	The terminal condition follows from the conjugacy of $H_+$ and $h_+$:
	$\widehat v^+(0,x)=h_+(x)$.  At $x=0$ the maximizer is $p=0$, whereas
value matching at the moving state boundary gives
	\[
	\widehat v^+(s,R_+(s))
	=\beta(s)R_+(s)-\mathcal W(s,\beta(s))=Q_+(s).
	\]
	By \eqref{eq:global-candidate}, $\widehat v(t,x)=\widehat
	v^+(s(t),x)$ for $x>0$. Using $R_+(s(t))=r(t)$ and
	$Q_+(s(t))=g_r(t)$, we conclude that
	\[
		\widehat v(t,0+)=0,\qquad
		\widehat v(t,r(t))=g_r(t),\qquad
		\widehat v(T,x)=h(x),\quad 0<x\le r(T).
	\]
	The preceding asymptotics near $x=0$ imply
	\[
	\widehat v(t,0+)
	=\widehat v_x(t,0+)
	=\widehat v_t(t,0+)=0.
	\]
	They also give
	\[
	\widehat v_{xx}(t,0+)
	=
	\begin{cases}
		h''(0)e^{-2s(t)},&m_+=2,\\
		0,&m_+>2,
	\end{cases}
	\]
	where in the first case $c_+=h''(0)$.
	
	The right-hand Hamiltonian minimizer can be written in dual variables as
	\begin{equation}\label{eq:right-feedback-dual-form}
		\widehat u^+(t,x)
		=-\frac{b_t}{\sigma_t^2}
		p(s(t),x)\mathcal W_{pp}\bigl(s(t),p(s(t),x)\bigr).
	\end{equation}
	The inverse $p(s,x)$ is smooth locally in the open primal domain.
	Consequently, $\widehat u^+$ is Borel measurable and locally
	Lipschitz in $x$ on compact subsets of $\mathcal D_+^\circ$.
	Lemma~\ref{lem:uniform-weighted-dual-second-derivative} gives
	\begin{equation}\label{eq:right-feedback-bound}
		|\widehat u^+(t,x)|
		\le
		\left\|\frac b{\sigma^2}\right\|_\infty
		\sup_{\substack{0\le s\le S\\0<p<\beta(s)}}
		p\,\mathcal W_{pp}(s,p)
		=:C_u<\infty.
	\end{equation}

	Applying the same argument to the reflected data gives the left branch,
	including its HJB equation, terminal and lateral data, and endpoint
	asymptotics. Thus the two branches have matching zero-order, first spatial,
	and time traces at $x=0$. If $h''(0)>0$, the $C^2$ regularity of $h$ forces
	$m_+=m_-=2$ and $c_+=c_-=h''(0)$; if $h''(0)=0$, both second-order traces
	vanish. Hence they paste into
	\[
		\widehat v\in C^{1,2}(\mathcal D^\circ)\cap C(\mathcal D),
		\qquad
		\widehat v_{xx}>0\quad(x\ne0),
	\]
	and \eqref{eq:global-interface-traces} holds.

	The HJB equation is already satisfied away from zero. If $h''(0)>0$, its
	Hamiltonian at $x=0$ uniquely selects $u=0$. If $h''(0)=0$, then
	$\widehat v_t=\widehat v_x=\widehat v_{xx}=0$ at zero, so the original
	Hamiltonian in the HJB equation in \eqref{eq:HJB-boundary-problem} vanishes
	there for every pointwise $u$. Thus $\widehat v$ satisfies the HJB equation
	at zero as well, and all the boundary--terminal data in
	\eqref{eq:HJB-boundary-problem} hold.

	The reflected version of \eqref{eq:right-feedback-bound} and the definition
	\eqref{eq:global-feedback} show that $\widehat u$ is bounded. Moreover, the
	first- and second-derivative asymptotics in
	\eqref{eq:primal-endpoint-asymptotics} give
	\[
		\widehat u(t,x)
		=-\frac{b_t}{\sigma_t^2(m_+-1)}x+o(x),
		\qquad x\downarrow0,
	\]
	and the reflected problem gives the corresponding expansion as
	$x\uparrow0$. Finally,
	\[
		\partial_x\bigl[p\mathcal W_{pp}(s,p)\bigr]
		=1+\frac{p\mathcal W_{ppp}(s,p)}{\mathcal W_{pp}(s,p)}
		\longrightarrow\frac1{m_+-1}.
	\]
	The $k=2,3$ asymptotics in \eqref{eq:W-general-power} therefore give
	bounded one-sided derivatives of $\widehat u$ near zero. Together with
	interior regularity, this proves the asserted local Lipschitz property and
	\eqref{eq:global-feedback-interface-regularity}.
\end{proof}

\subsection{Verification and optimality}
\label{subsec:global-pasting}

\begin{Theorem}[Identification with the value function and optimality]
	\label{thm:global-optimality}
	Suppose that Assumptions~\ref{ass:coefficients},
	\ref{ass:dual-comparison}, and~\ref{ass:finite-order} hold.
	Then the classical solution $\widehat v$ constructed in
	Theorem~\ref{thm:regional-reconstruction} coincides with the value function
	$V$ in \eqref{eq:value}. The feedback $u^*:=\widehat u$ from
	\eqref{eq:global-feedback}, stopped when the state exits the moving domain,
	is admissible and optimal. Consequently, all regularity,
	zero-volatility-set, and
	feedback properties established in Theorem~\ref{thm:regional-reconstruction}
	hold for $V$ and $u^*$.
\end{Theorem}

\begin{proof}
	By Theorem~\ref{thm:regional-reconstruction}, $\widehat u$ is bounded and
	locally Lipschitz in the state variable. Standard SDE localization therefore
	gives a unique strong closed-loop solution up to
	\[
		\tau^*
		:=\inf\{\theta\in[t,T]:
		X_\theta^*\notin(\ell(\theta),r(\theta))\}\wedge T,
	\]
	where, before $\tau^*$,
	\[
		dX_\theta^*
		=b_\theta\widehat u(\theta,X_\theta^*)\,d\theta
		+\sigma_\theta\widehat u(\theta,X_\theta^*)\,dW_\theta,
		\qquad X_t^*=x.
	\]
	Define the stopped feedback control by
	\begin{equation}\label{eq:stopped-optimal-control}
		u_\theta^*
		:=\widehat u(\theta,X_\theta^*)
		\mathbf 1_{\{\theta<\tau^*\}}.
	\end{equation}
	After $\tau^*$ the state is constant. Since $\widehat u$ is bounded,
	\[
		\mathbb E\int_t^T|u_\theta^*|^2\,d\theta<\infty,
	\]
	so $u^*$ is admissible.

	Fix an arbitrary $u\in\mathcal U(t,x)$, let $X$ be its state process, and
	write $\tau=\tau^{t,x;u}$. Choose compact sets
	$K_n\uparrow\mathcal D^\circ$ and set
	\[
		\tau_n
		:=\inf\{\theta\in[t,T]:(\theta,X_\theta)\notin K_n\}
		\wedge\tau\wedge(T-1/n).
	\]
	For all sufficiently large $n$, It\^o's formula gives
	\begin{equation*}
		\widehat v(t,x)
		=\mathbb E\widehat v(\tau_n,X_{\tau_n})-\mathbb E\int_t^{\tau_n}
		\left[
		\widehat v_t
		+b_\theta u_\theta\widehat v_x
		+\frac12\sigma_\theta^2u_\theta^2\widehat v_{xx}
		\right](\theta,X_\theta)\,d\theta.
	\end{equation*}
	The stochastic integral has zero expectation because the derivatives of
	$\widehat v$ are bounded on $K_n$ and $u$ is square integrable. Since
	$\widehat v$ satisfies the HJB equation in
	\eqref{eq:HJB-boundary-problem}, the integrand is nonnegative for every
	admissible $u$. Hence
	\[
		\widehat v(t,x)
		\le\mathbb E\widehat v(\tau_n,X_{\tau_n}).
	\]
	As $n\to\infty$, continuity of the state paths gives
	$\tau_n\uparrow\tau$ and $X_{\tau_n}\to X_\tau$. The function
	$\widehat v$ is continuous and bounded on $\mathcal D$, so dominated
	convergence and the boundary--terminal data in
	\eqref{eq:HJB-boundary-problem} yield
	\[
		\widehat v(t,x)\le J(t,x;u).
	\]

	For the stopped feedback \eqref{eq:stopped-optimal-control}, the selector
	$\widehat u$ attains the Hamiltonian minimum before $\tau^*$. Repeating the
	localized It\^o argument makes the integral vanish and gives
	\[
		\widehat v(t,x)=J(t,x;u^*).
	\]
	Taking the infimum over admissible controls proves
	\[
		\widehat v(t,x)=J(t,x;u^*)=V(t,x),
	\]
	and $u^*$ is optimal.
\end{proof}

\begin{Remark}[Uniqueness in the verification class]
\label{rem:classical-uniqueness}
The same argument gives uniqueness among functions
$f\in C^{1,2}(\mathcal D^\circ)\cap C(\mathcal D)$ that solve
\eqref{eq:HJB-boundary-problem} and admit a bounded Borel minimizing
feedback, locally Lipschitz in the state variable uniformly on compact
subsets of $\mathcal D^\circ$. Indeed, the stopped feedback generates
an admissible control. Localized It\^o's formula gives $f\le J(\cdot;u)$
for every admissible $u$, with equality for the minimizing feedback.
Continuity on the compact set $\mathcal D$ justifies passage to the
exit time, and hence $f=V$.
\end{Remark}

\begin{Corollary}[Spatial regularity at zero]
\label{cor:sharp-spatial-regularity}
Under the hypotheses of Theorem~\ref{thm:global-optimality}, if
$m_+\ne m_-$, the optimal feedback is not differentiable at $x=0$
for any $t<T$. If $m_+,m_->2$, set
\[
\beta:=\min\{1,m_+-2,m_--2\}.
\]
Then $V_{xx}(t,\cdot)$ is locally $\beta$-H\"older continuous across
zero, with constants uniform for $t$ in compact subsets of $[0,T)$.
If $\beta<1$, this exponent is sharp: for each fixed $t<T$,
$V_{xx}(t,\cdot)$ is not $\eta$-H\"older continuous at zero for
any $\eta\in(\beta,1]$.
\end{Corollary}

\begin{proof}
The one-sided derivatives of $u^*$ at zero are
$-b_t/[\sigma_t^2(m_\pm-1)]$ by
\eqref{eq:global-feedback-interface-regularity}. They differ when
$m_+\ne m_-$ because $b_t\ne0$.

For the second assertion, differentiating the inverse Legendre relation
on the right gives
\[
V_{xxx}(t,x)
=-\frac{\mathcal W_{ppp}(s(t),p)}
       {\mathcal W_{pp}(s(t),p)^3},
\qquad \mathcal W_p(s(t),p)=x.
\]
The $k=2,3$ estimates in \eqref{eq:W-general-power}, together with
\eqref{eq:inverse-endpoint-p}, imply
\[
V_{xxx}(t,x)
\sim c_+(m_+-1)(m_+-2)e^{-q_+s(t)}x^{m_+-3}
\qquad(x\downarrow0).
\]
Reflection gives the analogous absolute bound on the left. Thus
$|V_{xxx}(t,x)|\le C|x|^{\beta-1}$ near zero, uniformly on the
stated time intervals. Integrating on either side, and using
$V_{xx}(t,0)=0$ to compare points on opposite sides, proves the
H\"older estimate. If $\beta<1$, one side has $m_\pm-2=\beta$.
On that side \eqref{eq:primal-endpoint-asymptotics} gives
$V_{xx}(t,x)\sim C(t)|x|^\beta$ with $C(t)>0$, which excludes
every larger H\"older exponent.
\end{proof}

\section{Extension to the general affine model}
\label{sec:affine}

We now return to the general controlled diffusion introduced in
\eqref{eq:general-state}, with terminal penalty $h(Y_T-d)$.  By
Assumption~\ref{ass:coefficients}, $\sigma$ is bounded away from zero.  Set
\[
\bar u_s:=u_s+\frac{\gamma_s}{\sigma_s},
\qquad
\widetilde c_s:=c_s-\frac{b_s\gamma_s}{\sigma_s}.
\]
Then
\[
dY_s
=\bigl(a_sY_s+b_s\bar u_s+\widetilde c_s\bigr)\,ds
+\sigma_s\bar u_s\,dW_s.
\]
Define
\[
\phi(t):=\exp\left(\int_t^T a_\theta\,d\theta\right),
\qquad
\psi(t):=d-\int_t^T\phi(\theta)\widetilde c_\theta\,d\theta,
\]
and
\begin{equation}
	\label{eq:affine-transformed-state}
	Z_s:=\phi(s)Y_s-\psi(s).
\end{equation}
Since $\phi'=-a\phi$ and $\psi'=\phi\widetilde c$, It\^o's formula gives
\begin{equation}
	\label{eq:affine-homogeneous-state}
	dZ_s
	=\bar b_s\bar u_s\,ds+\bar\sigma_s\bar u_s\,dW_s,
	\qquad
	\bar b_s:=\phi(s)b_s,\quad
	\bar\sigma_s:=\phi(s)\sigma_s.
\end{equation}
Moreover, $Z_T=Y_T-d$, so the terminal costs agree.

The transformed boundaries are
\[
\bar\ell(t):=\phi(t)\ell(t)-\psi(t),
\qquad
\bar r(t):=\phi(t)r(t)-\psi(t).
\]
Because $\phi>0$, the transformation preserves the order of the state and
the boundaries.  Corresponding state paths therefore have the same exit time
and corresponding boundary events, with the same boundary costs.  Denote by
$\bar J$ and $\bar V$ the cost functional and value function of the
homogeneous problem \eqref{eq:affine-homogeneous-state} on
$(\bar\ell,\bar r)$.

\begin{Theorem}[Value function and feedback in the affine model]
	\label{thm:affine-value}
	Define
	\begin{equation}\label{eq:affine-degenerate-curve}
		y^*(t):=\frac{\psi(t)}{\phi(t)}.
	\end{equation}
	Suppose that
	\[
	\ell(t)<y^*(t)<r(t),
	\qquad 0\le t\le T,
	\]
	and that the transformed homogeneous data satisfy the standing model
	conditions and Assumptions~\ref{ass:coefficients},
	\ref{ass:dual-comparison}, and~\ref{ass:finite-order}.  Then
	\begin{equation}\label{eq:affine-value-representation}
		V^{\rm gen}(t,y)
		=
		\bar V\bigl(t,\phi(t)y-\psi(t)\bigr).
	\end{equation}
	In particular, $V^{\rm gen}\in C^{1,2}$ in the interior of the original
	moving domain, and
	\begin{equation}\label{eq:affine-value-derivatives}
			V_y^{\rm gen}(t,y)=
			\phi(t)\bar V_z\bigl(t,\phi(t)y-\psi(t)\bigr),\quad
			V_{yy}^{\rm gen}(t,y)=
			\phi(t)^2\bar V_{zz}\bigl(t,\phi(t)y-\psi(t)\bigr).
	\end{equation}
	The optimal feedback is
	\begin{equation}\label{eq:affine-optimal-control}
		u^*(t,y)
		=
		\bar u^*\bigl(t,\phi(t)y-\psi(t)\bigr)
		-\frac{\gamma_t}{\sigma_t}.
	\end{equation}
	On the curve $y=y^*(t)$,
	\begin{equation}\label{eq:affine-interface-identities}
		V_y^{\rm gen}(t,y^*(t))=0,
		\qquad
		V_{yy}^{\rm gen}(t,y^*(t))
		=
		\phi(t)^2\bar V_{zz}(t,0),
		\qquad
		u^*(t,y^*(t))
		=
		-\frac{\gamma_t}{\sigma_t}.
	\end{equation}
	Moreover, as $y\to y^*(t)^\pm$,
	\begin{equation}\label{eq:affine-feedback-interface-asymptotics}
		u^*(t,y)+\frac{\gamma_t}{\sigma_t}
		=
		-\frac{b_t}{\sigma_t^2(m_\pm-1)}
		\bigl(y-y^*(t)\bigr)
		+o\bigl(|y-y^*(t)|\bigr),
	\end{equation}
	locally uniformly for $t\in[0,T)$.
\end{Theorem}

\begin{proof}
	Since $\gamma/\sigma$ is bounded, the shift \(	\bar u=u+\gamma/\sigma\)
	is a bijection between the two admissible control classes.  The state
	transformation preserves exit times and boundary events, and
	$Z_T=Y_T-d$.  Hence, under corresponding controls,
	\[
	J^{\rm gen}(t,y;u)
	=
	\bar J\bigl(t,\phi(t)y-\psi(t);\bar u\bigr).
	\]
	Taking infima proves \eqref{eq:affine-value-representation}.  Since
	$\bar V$ is the classical value function of the transformed problem, the
	chain rule gives \eqref{eq:affine-value-derivatives}; shifting the transformed
	optimal feedback back gives \eqref{eq:affine-optimal-control}.
	
	By definition, \(	\phi(t)y^*(t)-\psi(t)=0\).
The traces at $z=0$ and the identity $\bar u^*(t,0)=0$ from
	Theorem~\ref{thm:global-optimality} therefore give
	\eqref{eq:affine-interface-identities}.  Finally, apply
	\eqref{eq:global-feedback-interface-regularity} with
	\[
	z=\phi(t)\bigl(y-y^*(t)\bigr),
	\qquad
	\bar b_t=\phi(t)b_t,
	\qquad
	\bar\sigma_t=\phi(t)\sigma_t,
	\]
	and then use \eqref{eq:affine-optimal-control}.  This yields
	\eqref{eq:affine-feedback-interface-asymptotics}.
\end{proof}

Theorem~\ref{thm:affine-value} shows that the vertical line $z=0$ in the
homogeneous coordinates corresponds to a generally nonzero and time-dependent
curve in the original model. Thus the fixed set $z=0$ in the transformed
problem is a consequence of the chosen coordinates rather than a restriction
on the original dynamics.

\section{Optimal trajectories and approach probabilities near the zero-volatility set}
\label{sec:probability}

A pointwise vanishing diffusion coefficient does not by itself imply that a
trajectory starting there remains there or that the point cannot be reached
from elsewhere.  For example, the uncontrolled equation
\[
dY_\theta=(Y_\theta-1)\,d\theta+Y_\theta\,dW_\theta
\]
has zero diffusion coefficient at $Y=0$, but its drift there is $-1$; a
solution starting from zero therefore leaves zero immediately.  The behavior
of the zero-volatility set identified above must instead be derived from the
optimal closed-loop equation.

Let $X^*$ be the closed-loop state generated by the optimal feedback
$u^*=\widehat u$ identified in Theorem~\ref{thm:global-optimality} and defined
in \eqref{eq:global-feedback}, and set
\[
\tau^X:=
\inf\{\theta\in[t,T]:
X_\theta^*\notin(\ell(\theta),r(\theta))\}\wedge T,
\qquad
\tau_0:=\inf\{\theta\in[t,T]:X_\theta^*=0\},
\]
with $\inf\varnothing=\infty$.  Suppose first that $x>0$ and define
\[
\Pi_\theta:=V_x(\theta,X_\theta^*),
\qquad \theta<\tau_0\wedge\tau^X.
\]
After localization on compact subsets of the right one-sided interior, the
inverse-dual representation and interior parabolic regularity justify
It\^o's formula for $V_x$.  Differentiating the reduced HJB equation and
substituting the optimal feedback show that the drift terms cancel, giving
\begin{equation}\label{eq:marginal-SDE}
	d\Pi_\theta
	=-\frac{b_\theta}{\sigma_\theta}\Pi_\theta\,dW_\theta,
	\qquad
	X_\theta^*=\mathcal W_p(s(\theta),\Pi_\theta).
\end{equation}
Consequently,
\begin{equation}\label{eq:marginal-exponential}
	\Pi_\theta
	=\Pi_t\exp\!\left\{
	-\int_t^\theta\frac{b_\xi}{\sigma_\xi}\,dW_\xi
	-\frac12\int_t^\theta
	\left(\frac{b_\xi}{\sigma_\xi}\right)^2d\xi
	\right\}.
\end{equation}
The reflected identity holds on the left.  Thus the value gradient is not
only an adjoint process: it closes as a geometric martingale, while the
inverse relation in \eqref{eq:marginal-SDE} recovers the state variable from
the marginal value.

\begin{Theorem}[Behavior at the zero-volatility set]
	\label{thm:interface-inaccessibility}
	Under the hypotheses of Theorem~\ref{thm:global-optimality}, fix $(t,x)$ and
	assume that
	\[
	\ell(\theta)<0<r(\theta),
	\qquad \theta\in[t,T].
	\]
	If $x\ne0$, then
	\begin{equation}\label{eq:interface-inaccessible}
		\mathbb P_{t,x}(\tau_0<\tau^X)=0,
	\end{equation}
	and the optimal state preserves its sign before leaving the state interval.
	If $x=0$,
	then
	\begin{equation}\label{eq:interface-invariant}
		X_\theta^*=0,
		\qquad \theta\in[t,T],
		\qquad\text{a.s.}
	\end{equation}
	
	The same conclusions hold for the affine model.  More precisely, suppose
	that the hypotheses of Theorem~\ref{thm:affine-value} hold and that
	\[
	\ell(\theta)<y^*(\theta)<r(\theta),
	\qquad \theta\in[t,T].
	\]
	For the affine optimal state $Y^*$, define
	\[
	\tau^Y:=
	\inf\{\theta\in[t,T]:
	Y_\theta^*\notin(\ell(\theta),r(\theta))\}\wedge T,
	\qquad
	\tau_\Gamma:=
	\inf\{\theta\in[t,T]:
	Y_\theta^*=y^*(\theta)\}.
	\]
	If $y\ne y^*(t)$, then
	\[
	\mathbb P_{t,y}(\tau_\Gamma<\tau^Y)=0,
	\]
	and $Y_\theta^*-y^*(\theta)$ preserves its sign before $\tau^Y$.
	If $y=y^*(t)$, then
	\[
	Y_\theta^*=y^*(\theta),
	\qquad \theta\in[t,T],
	\qquad\text{a.s.}
	\]
\end{Theorem}

\begin{proof}
	Suppose first that $x>0$.  The stochastic exponential in
	\eqref{eq:marginal-exponential} has a strictly positive limit at every
	bounded stopping time.  On $\{\tau_0<\tau^X\}$, however, continuity of
	$X^*$ and $V_x$, together with $V_x(\tau_0,0)=0$, would imply
	$\Pi_\theta\to0$ as $\theta\uparrow\tau_0$, a contradiction.  The reflected
	argument treats $x<0$.  If $x=0$, the zero path is a closed-loop solution
	because $u^*(\theta,0)=0$; the local Lipschitz property of the feedback and
	pathwise uniqueness show that it is the unique solution.
	
	For the affine problem,
	\[
	Z_\theta^*
	=\phi(\theta)Y_\theta^*-\psi(\theta)
	=\phi(\theta)\bigl(Y_\theta^*-y^*(\theta)\bigr).
	\]
	Since $\phi>0$, hitting the affine curve is equivalent to $Z^*$ hitting
	zero, and the two differences have the same sign.  The homogeneous
	conclusions therefore transfer directly.  On the curve,
	\[
	\sigma_\theta u^*(\theta,y^*(\theta))+\gamma_\theta=0,
	\qquad
	a_\theta y^*(\theta)
	+b_\theta u^*(\theta,y^*(\theta))+c_\theta
	=(y^*)'(\theta),
	\]
	so the optimal volatility vanishes and the closed-loop drift equals the
	time derivative of the curve.  Thus a trajectory starting on the curve
	follows it.
\end{proof}

The proof of Theorem~\ref{thm:approach-rates} requires a Brownian terminal
lower tail under a bounded lower-semicontinuous barrier. Related boundary
crossing formulas and asymptotics appear in
\cite{Shepp1979,WangPotzelberger1997,BeibelLerche1994,KlumpSavov2025}; the
following lemma records the precise form needed here.

\begin{Lemma}[Terminal lower tail below a lower-semicontinuous barrier]
	\label{lem:brownian-lower-tail}
	Let
	\[
	Y_a=y_0+B_a-\frac a2,
	\qquad 0\le a\le A,
	\]
	where $A>0$ and $B$ is a standard Brownian motion.  If
	$g:[0,A]\to\mathbb R$ is bounded and lower semicontinuous and
	$y_0<g(0)$, then
	\begin{equation}\label{eq:brownian-lower-tail}
		\lim_{L\to\infty}\frac{1}{L^2}
		\log\mathbb P
		\left\{
		Y_A\le-L,\quad
		Y_a<g(a)\text{ for every }a\in[0,A]
		\right\}
		=-\frac{1}{2A}.
	\end{equation}
\end{Lemma}

\begin{proof}
	After dropping the path constraint, the Gaussian lower-tail asymptotic gives
	the upper bound.
	
	For the reverse bound, choose
	\[
	H<\inf_{0\le a\le A}g(a)
	\]
	and a compact interval $I\Subset(-\infty,H)$.  For every sufficiently small
	$a_0>0$, lower semicontinuity at zero and $y_0<g(0)$ allow one to construct
	a continuous path joining $y_0$ to the interior of $I$ whose graph lies
	strictly below $g$ on $[0,a_0]$.  Its graph is compact in the open set
	$\{(a,y):y<g(a)\}$ and therefore admits a uniform tube contained in that
	set.  Full support of Wiener measure gives a constant $c_0(a_0)>0$,
	independent of $L$, such that
	\[
	\mathbb P\left\{
	Y_a<g(a),\ 0\le a\le a_0,\ 
	Y_{a_0}\in I
	\right\}
	\ge c_0(a_0).
	\]
	
	Put $D:=A-a_0$ and $\mu:=-1/2$.  Integrating the constant-level
	maximum--endpoint law in \cite{Shepp1979}, or equivalently applying the
	reflection principle, gives the transition density of Brownian motion with
	drift $\mu$ killed upon reaching $H$:
	\[
	p_D^H(z,y)
	=
	\frac{e^{\mu(y-z)-\mu^2D/2}}{\sqrt{2\pi D}}
	\left[
	e^{-(y-z)^2/(2D)}
	-
	e^{-(y+z-2H)^2/(2D)}
	\right],
	\qquad z,y<H.
	\]
	For $z\in I$ and $y\in[-L-1,-L]$, the ratio of the reflected Gaussian term
	to the first one is
	\[
	\exp\left\{-\frac{2(H-y)(H-z)}{D}\right\}
	=O(e^{-cL})
	\]
	uniformly in $z$.  Consequently,
	\[
	\inf_{z\in I}
	\log\int_{-L-1}^{-L}p_D^H(z,y)\,dy
	=
	-\frac{L^2}{2D}+O(L).
	\]
	The Markov property at $a_0$, together with $H<g$, therefore yields
	\[
	\liminf_{L\to\infty}\frac{1}{L^2}
	\log\mathbb P
	\left\{
	Y_A\le-L,\ 
	Y_a<g(a)\text{ for every }a\in[0,A]
	\right\}
	\ge-\frac{1}{2(A-a_0)}.
	\]
	Letting $a_0\downarrow0$ proves the lower bound.
\end{proof}

\begin{Theorem}[Approach rates in marginal and state coordinates]
	\label{thm:approach-rates}
	Under the hypotheses of Theorem~\ref{thm:global-optimality}, fix
	$t<t_1<T$ and $0<x<r(t)$.  For all sufficiently small
	$0<\varepsilon<r(t_1)$, set
	\begin{equation}\label{eq:variance-clock-t1}
		A_{t,t_1}
		:=
		\int_t^{t_1}
		\left(\frac{b_\theta}{\sigma_\theta}\right)^2d\theta>0,
		\qquad
		p_\varepsilon:=V_x(t_1,\varepsilon).
	\end{equation}
	Then
	\begin{equation}\label{eq:universal-interface-approach}
		\lim_{\varepsilon\downarrow0}
		\frac{
			\log\mathbb P_{t,x}
			\{0<X_{t_1}^*\le\varepsilon,\ t_1<\tau^X\}
		}{
			[\log(1/p_\varepsilon)]^2
		}
		=
		-\frac{1}{2A_{t,t_1}}.
	\end{equation}
	If the right terminal minimum has order $m_+$, then
	\begin{equation}\label{eq:interface-approach-rate}
		\lim_{\varepsilon\downarrow0}
		\frac{
			\log\mathbb P_{t,x}
			\{0<X_{t_1}^*\le\varepsilon,\ t_1<\tau^X\}
		}{
			[\log(1/\varepsilon)]^2
		}
		=
		-\frac{(m_+-1)^2}{2A_{t,t_1}}.
	\end{equation}
	If instead $\ell(t)<x<0$, the reflected statements hold with
	$p_\varepsilon^-:=-V_x(t_1,-\varepsilon)$ and the event
	$\{-\varepsilon\le X_{t_1}^*<0,\ t_1<\tau^X\}$ in place of their
	right-hand counterparts, and with $m_-$ in place of $m_+$.
\end{Theorem}

\begin{proof}
	We prove the right-hand assertions. Put $p:=V_x(t,x)$ and let
	$\widehat\Pi$ be the stochastic exponential on $[t,t_1]$ given by the
	right-hand side of \eqref{eq:marginal-exponential}, with initial value
	$\widehat\Pi_t=p$. Let
	\[
	E_\varepsilon
	:=
	\left\{
	\widehat\Pi_{t_1}\le p_\varepsilon,\quad
	\widehat\Pi_\theta<\beta(s(\theta))
	\text{ for every }\theta\in[t,t_1]
	\right\}.
	\]
	We first justify the corresponding safe-path implication. Let
	\[
		\mathcal C_R:=\{(s,\pi):0<\pi<\beta(s)\},\qquad
		\eta:=\inf\{\theta\in[t,t_1]:
		(s(\theta),\widehat\Pi_\theta)\notin\mathcal C_R\},
	\]
	with $\inf\varnothing=\infty$. For $\theta<\eta$, set
	\[
		\widehat X_\theta
		:=\mathcal W_p(s(\theta),\widehat\Pi_\theta),
		\qquad
		\widehat u_\theta
		:=-\frac{b_\theta}{\sigma_\theta^2}
		\widehat\Pi_\theta
		\mathcal W_{pp}(s(\theta),\widehat\Pi_\theta).
	\]
	Because $\beta$ is lower semicontinuous, $\mathcal C_R$ is open. On
	$E_\varepsilon$, the graph of $(s(\theta),\widehat\Pi_\theta)$,
	$t\le\theta\le t_1$, is a compact subset of $\mathcal C_R$, and hence
	$\eta>t_1$. Using
	$\mathcal W_{sp}=2p\mathcal W_{pp}+p^2\mathcal W_{ppp}$, It\^o's formula
	and the inverse Legendre identities show that $\widehat X$ satisfies the
	optimal closed-loop SDE on
	$[t,t_1]$. Moreover, $\widehat X_t=x$. Pathwise uniqueness therefore gives
	\[
		t_1<\tau^X,\qquad
		X_{t_1}^*
		=\mathcal W_p(s(t_1),\widehat\Pi_{t_1})
		\quad\text{on }E_\varepsilon.
	\]
	Consequently, monotonicity of $\mathcal W_p$ and
	\[
	\mathcal W_p(s(t_1),p_\varepsilon)=\varepsilon,
	\]
	show that the event $E_\varepsilon$ is contained in
	\[
	\{0<X_{t_1}^*\le\varepsilon,\ t_1<\tau^X\}.
	\]
	Conversely, on $\{t_1<\tau^X\}$ the actual marginal process and
	$\widehat\Pi$ solve the same linear stochastic equation with the same initial
	value and hence coincide.  Strict convexity therefore gives
	\begin{equation}\label{eq:approach-sandwich}
		\mathbb P(E_\varepsilon)
		\le
		\mathbb P_{t,x}
		\{0<X_{t_1}^*\le\varepsilon,\ t_1<\tau^X\}
		\le
		\mathbb P\{\widehat\Pi_{t_1}\le p_\varepsilon\}.
	\end{equation}
	
	Introduce the deterministic time change
	\[
	a(\theta)
	:=
	\int_t^\theta
	\left(\frac{b_\xi}{\sigma_\xi}\right)^2d\xi,
	\qquad
	A:=a(t_1)=A_{t,t_1},
	\]
	and denote its inverse by $\theta=\theta(a)$.
	Assumption~\ref{ass:coefficients} makes this time change strictly increasing.
	For a standard
	Brownian motion $B$,
	\begin{equation}\label{eq:log-marginal-clock}
		Y_a
		:=
		\log\widehat\Pi_{\theta(a)}
		=
		\log p+B_a-\frac a2,
		\qquad 0\le a\le A.
	\end{equation}
	Set
	\[
	g(a):=\log \beta(s(\theta(a))).
	\]
	The uniform positive bounds and lower semicontinuity of $\beta$ imply that
	$g$ is bounded and lower semicontinuous.  Moreover,
	$\log p<g(0)$ because $p<\beta(s(t))$.  Since
	$p_\varepsilon\downarrow0$, with
	\[
	L_\varepsilon:=\log(1/p_\varepsilon)
	\]
	one has exactly
	\[
	E_\varepsilon
	=
	\left\{
	Y_A\le-L_\varepsilon,\quad
	Y_a<g(a)\text{ for every }a\in[0,A]
	\right\}.
	\]
	Lemma~\ref{lem:brownian-lower-tail} gives the logarithmic limit
	$-1/(2A_{t,t_1})$ for the left-hand side of
	\eqref{eq:approach-sandwich}; the ordinary Gaussian lower-tail asymptotic
	gives the same limit for the right-hand side.  This proves
	\eqref{eq:universal-interface-approach}.
	
	Finally, let $q_+:=m_+/(m_+-1)$.  The $k=1$ case of
	\eqref{eq:W-general-power}, evaluated at
	$\mathcal W_p(s(t_1),p_\varepsilon)=\varepsilon$, gives
	\[
	p_\varepsilon
	=
	c_+e^{-q_+s(t_1)}
	\varepsilon^{m_+-1}(1+o(1)).
	\]
	Thus
	\[
	\log(1/p_\varepsilon)
	=
	(m_+-1)\log(1/\varepsilon)+O(1),
	\]
	and \eqref{eq:interface-approach-rate} follows from
	\eqref{eq:universal-interface-approach}.  The left-hand assertions follow
	from the reflected construction.
\end{proof}

\begin{Example}[An explicit asymmetric model]
	\label{ex:asymmetric-approach}
	Let \(b\equiv\sigma\equiv1\), let the state interval be \((-1,1)\), and
	take
	\[
		s(t)=\frac{T-t}{2},
		\qquad
		h(x)=
		\begin{cases}
			x^4/4,&x\ge0,\\
			|x|^6/6,&x<0.
		\end{cases}
	\]
	Then \(h\in C^2(\mathbb R)\) is strictly convex with
	\[
		m_+=4,\quad m_-=6,\qquad
		q_+=\frac43,\quad q_-=\frac65.
	\]
	For \(h_m(x)=x^m/m\), direct conjugacy gives
	\[
		H_{\infty,m}(p)=\frac{p^q}{q},\qquad
		\overline p_m(s)=e^{-qs},
		\quad q=\frac{m}{m-1}.
	\]
	Thus the lateral costs
	\[
		g_r(t)=\frac14e^{-2(T-t)/3},
		\qquad
		g_\ell(t)=\frac16e^{-3(T-t)/5}
	\]
	satisfy
	\[
	-Q_\pm'(s)
	=(\mathscr H_\pm)_s(s,\overline p_\pm(s)),
	\]
	so Assumption~\ref{ass:dual-comparison} holds with equality, while
	Assumption~\ref{ass:finite-order} holds exactly.

	Direct substitution and verification give
	\[
	V(t,x)=
	\begin{cases}
		\dfrac16e^{-6s(t)/5}|x|^6,&-1\le x<0,\\[5pt]
		0,&x=0,\\[3pt]
		\dfrac14e^{-4s(t)/3}x^4,&0<x\le1,
	\end{cases}
	\qquad
	u^*(t,x)=
	\begin{cases}
		-x/5,&x<0,\\
		0,&x=0,\\
		-x/3,&x>0.
	\end{cases}
	\]
	Hence \(V\) is strictly convex although \(V_{xx}(t,0)=0\), and the
	feedback has different one-sided linear slopes. Since
	\(A_{t,t_1}=t_1-t\), Theorem~\ref{thm:approach-rates} yields
	\[
	\lim_{\varepsilon\downarrow0}
	\frac{\log\mathbb P_{t,x}
		\{0<X_{t_1}^*\le\varepsilon,\ t_1<\tau^X\}}
	{[\log(1/\varepsilon)]^2}
	=-\frac9{2(t_1-t)},\qquad x>0,
	\]
	and
	\[
	\lim_{\varepsilon\downarrow0}
	\frac{\log\mathbb P_{t,x}
		\{-\varepsilon\le X_{t_1}^*<0,\ t_1<\tau^X\}}
	{[\log(1/\varepsilon)]^2}
	=-\frac{25}{2(t_1-t)},\qquad x<0.
	\]
	Thus the same inaccessible zero-volatility set has different approach
	rates on its two sides.
\end{Example}

\begin{Example}[A strict perturbation of the lateral costs]
\label{ex:boundary-perturbation}
Keep the coefficients, interval, and terminal loss of
Example~\ref{ex:asymmetric-approach}, but replace the lateral costs by
\[
Q_{\delta,\pm}(s):=\frac1{m_\pm}e^{-q_\pm s}-\delta_\pm s,
\qquad
0<\delta_\pm<\frac{e^{-q_\pm S}}{m_\pm S},
\qquad S=T/2.
\]
Here $g_{\delta,r}(t)=Q_{\delta,+}(s(t))$ and
$g_{\delta,\ell}(t)=Q_{\delta,-}(s(t))$.
These costs are positive and have the same terminal values as before.
Since $R_\pm=1$ and
\[
-Q_{\delta,\pm}'(s)
=\frac{e^{-q_\pm s}}{m_\pm-1}+\delta_\pm
=(\mathscr H_\pm)_s(s,\overline p_\pm(s))+\delta_\pm,
\]
Assumption~\ref{ass:dual-comparison} holds with a strict margin.

Let $V_\delta$ denote the new value function and write
$F_\pm(t,x):=e^{-q_\pm s(t)}|x|^{m_\pm}/m_\pm$
on the corresponding side. Then
\begin{equation}\label{eq:perturbed-value-strict}
0<V_\delta(t,x)<F_\pm(t,x),
\qquad t<T,\quad 0<\pm x<1.
\end{equation}
To see the strict upper bound, use the unperturbed feedback
$u^0=-X/(m_\pm-1)$ up to its exit time $\tau^0$, capped by $T$.
The process $F_\pm(\theta,X_\theta)$, stopped at $\tau^0$, is a
bounded martingale. Consequently,
\[
J_\delta(t,x;u^0)
=F_\pm(t,x)-\delta_\pm
\mathbb E_{t,x}^{u^0}\!\left[
s(\tau^0)\mathbf1_{\{\tau^0<T\}}\right]
<F_\pm(t,x).
\]
The final inequality holds because this geometric diffusion has a
positive probability of reaching the corresponding endpoint before
$T$. Taking the infimum gives the upper bound in
\eqref{eq:perturbed-value-strict}; the lower bound follows from
strict convexity and $V_\delta(t,0)=(V_\delta)_x(t,0)=0$.

Nevertheless, \eqref{eq:primal-endpoint-asymptotics} gives
\[
V_\delta(t,x)\sim F_\pm(t,x),\qquad
(V_\delta)_{xx}(t,x)\sim
(m_\pm-1)e^{-q_\pm s(t)}|x|^{m_\pm-2}
\quad(x\to0^\pm).
\]
Thus the exit boundary changes the value without changing its leading
behavior near zero. Nor can the perturbed optimal feedback remain
$-x/(m_\pm-1)$ on an entire side: the feedback formula would then
force $V_\delta(t,x)=C_\pm(t)|x|^{m_\pm}$, and the asymptotics
would identify it with $F_\pm$, contradicting
\eqref{eq:perturbed-value-strict}.

The same construction is valid with $m_+=5/2$ and $m_-=7/2$.
In that case Corollary~\ref{cor:sharp-spatial-regularity} gives a
classical value function whose second spatial derivative has the
sharp H\"older exponent $1/2$ at zero, and an optimal feedback
that is locally Lipschitz but not differentiable there.
\end{Example}

\appendix
\numberwithin{equation}{section}
\section{State intervals lying on one side of zero}
\label{app:one-sided-domain}

Assume $0<\ell(t)<r(t)$; the negative one-sided case follows by reflection.
In backward time write
\[
L=\ell\circ t,\qquad R=r\circ t,\qquad
Q_\ell=g_\ell\circ t,\qquad Q_r=g_r\circ t.
\]
For this appendix, assume that the terminal loss prescribed on
$[L(0),R(0)]$ extends to $[0,R(0)]$ as a one-sided terminal loss satisfying
the conditions stated for the homogeneous model and admitting an admissible
extension to $[0,\infty)$.  We continue to denote the extension to
$[0,R(0)]$ by $h$.

Set
\[
\begin{aligned}
\Phi_\ell(s,p)&:=pL(s)-Q_\ell(s),
&\Phi_r(s,p)&:=pR(s)-Q_r(s),\\
\Psi(s,p)&:=\max\{\Phi_\ell(s,p),\Phi_r(s,p)\}.
\end{aligned}
\]
The terminal reward for the dual problem is
\[
H_{LR}(p)
:=
\sup_{L(0)\le x\le R(0)}\{px-h(x)\},
\qquad p\ge0.
\]

Fix an admissible extension $\widehat h$ to $[0,\infty)$ and define
\[
H_\infty(p):=\sup_{x\ge0}\{px-\widehat h(x)\},
\qquad
\mathscr H(s,p):=\mathbb E[H_\infty(pZ_s)].
\]
Since $\mathscr H_{pp}>0$, there are unique continuous functions
$\underline p,\overline p$ satisfying
\[
\mathscr H_p(s,\underline p(s))=L(s),
\qquad
\mathscr H_p(s,\overline p(s))=R(s).
\]
We impose the sufficient two-obstacle comparison condition
\begin{equation}\label{eq:one-sided-comparison}
	L'\le0\le R',
	\qquad
	\Delta
	:=\underline pL-\mathscr H(s,\underline p)-Q_\ell
	=\overline pR-\mathscr H(s,\overline p)-Q_r,
	\qquad
	\Delta'\ge0.
\end{equation}
Terminal-corner compatibility gives $\Delta(0)=0$.  Requiring the two gaps to
coincide is sufficient, but is not necessary for the original control problem.

Let
\[
\mathcal W(s,p)
:=
\sup_{0\le\tau\le s}
\mathbb E\!\left[
\Psi(s-\tau,P_\tau^p)\mathbf 1_{\{\tau<s\}}
+H_{LR}(P_s^p)\mathbf 1_{\{\tau=s\}}
\right],
\]
and write
\[
\mathcal S_s
:=
\{p\ge0:\mathcal W(s,p)=\Psi(s,p)\}.
\]

\begin{Proposition}
	\label{prop:one-sided-extension}
	Under Assumption~\ref{ass:coefficients} and the preceding one-sided
	extension and comparison conditions, there are two unique stopping
	boundaries $\beta_\ell,\beta_r$ such that
	\[
	\mathcal S_s
	=
	[0,\beta_\ell(s)]\cup[\beta_r(s),\infty).
	\]
	Moreover, with
	\[
	m_0(s)
	:=
	\frac{Q_r(s)-Q_\ell(s)}{R(s)-L(s)},
	\]
	one has
	\begin{equation}\label{eq:one-sided-boundary-order}
		0<\underline p(s)
		\le\beta_\ell(s)<m_0(s)
		<\beta_r(s)\le\overline p(s),
		\qquad 0\le s\le S.
	\end{equation}
	In particular, the continuation region is uniformly separated from $p=0$.
	The inverse Legendre transform of $\mathcal W$ is the value function, belongs
	to $C^{1,2}(\mathcal D^\circ)\cap C(\mathcal D)$, is strictly convex in the
	state variable, and satisfies $V_x>0$ throughout $\mathcal D^\circ$.
	Consequently, the optimal volatility does not vanish in the interior.
\end{Proposition}

\begin{proof}
	Define
	\[
	\overline{\mathcal W}(s,p)
	:=
	\begin{cases}
		\Phi_\ell(s,p),&p\le\underline p(s),\\
		\mathscr H(s,p)+\Delta(s),
		&\underline p(s)<p<\overline p(s),\\
		\Phi_r(s,p),&p\ge\overline p(s).
	\end{cases}
	\]
	The definition of \(\Delta\) gives value matching, while those of
	\(\underline p\) and \(\overline p\) give slope matching.
	Hence the two outer branches are tangent to the strictly convex middle
	branch, their intersection satisfies
	\[
		\underline p<m_0<\overline p,
	\]
	and \(\overline{\mathcal W}\) is convex, \(C^1\) in \(p\), and dominates
	\(\Psi\). Terminal compatibility also gives
	\(\overline{\mathcal W}(0,\cdot)=H_{LR}\).

	Using \(\mathscr H_s=p^2\mathscr H_{pp}\), one obtains
	\[
	(\partial_s-p^2\partial_{pp})\overline{\mathcal W}
	=
	\begin{cases}
		(p-\underline p)L'
		+\mathscr H_s(s,\underline p)+\Delta',
		&p<\underline p,\\
		\Delta',&\underline p<p<\overline p,\\
		(p-\overline p)R'
		+\mathscr H_s(s,\overline p)+\Delta',
		&p>\overline p.
	\end{cases}
	\]
	All three branches are nonnegative by
	\eqref{eq:one-sided-comparison}. Generalized It\^o's formula produces no
	local-time term because the spatial slopes match, and optional sampling
	gives
	\[
		\Psi\le\mathcal W\le\overline{\mathcal W}.
	\]
	Thus immediate stopping is optimal for
	\(p\le\underline p\) and \(p\ge\overline p\).

	It remains to exclude contact where the two obstacles meet. For \(s>0\),
	the martingale property of \(P^{m_0(s)}\) and the kink of \(\Psi\) give
	\[
	\mathbb E\!\left[
		\Psi(s-\varepsilon,P_\varepsilon^{m_0(s)})
		\right]
	=
	\Psi(s,m_0(s))
	+\frac{m_0(s)(R(s)-L(s))}{\sqrt\pi}\sqrt\varepsilon
	+o(\sqrt\varepsilon).
	\]
	Hence waiting is strictly preferable at \(m_0(s)\). At \(s=0\), terminal
	compatibility and strict convexity imply
	\[
		h'(L(0))<m_0(0)<h'(R(0)),
		\qquad
		H_{LR}(m_0(0))>\Psi(0,m_0(0)).
	\]
	Thus \(m_0(s)\) lies in the continuation region for every \(s\in[0,S]\).

	The continuity argument from Proposition~\ref{prop:stopping-boundary}
	applies because \(H_{LR}\ge\Psi(0,\cdot)\). Moreover, coupling the same
	stopping time for \(p_2>p_1\) yields
	\[
	L(s)(p_2-p_1)
	\le\mathcal W(s,p_2)-\mathcal W(s,p_1)
	\le R(s)(p_2-p_1).
	\]
	Therefore \(\mathcal W-\Phi_\ell\) is nondecreasing and
	\(\mathcal W-\Phi_r\) is nonincreasing in \(p\). Together with the two
	contact on both outer intervals and strict continuation at \(m_0\), this proves
	\eqref{eq:one-sided-boundary-order} and the asserted two-boundary
	structure.

	Finally, \eqref{eq:one-sided-boundary-order} separates the continuation
	region uniformly from \(p=0\). The smooth-fit, local obstacle-regularity,
	and second-derivative estimates used for
	Proposition~\ref{prop:stopping-boundary},
	Lemma~\ref{lem:uniform-weighted-dual-second-derivative}, and
	Corollary~\ref{cor:strict-convexity} apply at the two free boundaries. With
	\[
		\mathcal C:=\{(s,p):\beta_\ell(s)<p<\beta_r(s)\},
	\]
	they give
	\[
		\sup_{\mathcal C}p\mathcal W_{pp}<\infty,
		\qquad
		\mathcal W_{pp}>0\quad\text{in }\mathcal C.
	\]
	The inverse-Legendre and verification arguments of
	Theorems~\ref{thm:regional-reconstruction} and
	\ref{thm:global-optimality} then yield the stated classical value function,
	\(V_x>0\), and the optimal feedback
	\[
		u^*(t,x)
		=-\frac{b_t}{\sigma_t^2}
		p\mathcal W_{pp}(s(t),p),
		\qquad \mathcal W_p(s(t),p)=x.
	\]
	Since \(p>0\) and \(\mathcal W_{pp}>0\) in the continuation region, the
	optimal volatility is nonzero throughout the state-domain interior.
\end{proof}

\section{Proof of Proposition~\ref{prop:stopping-boundary}}
\label{app:one-obstacle-boundary}

\begin{proof}
	Write
	\[
	Z_\theta:=e^{\sqrt{2}B_\theta-\theta},
	\qquad
	P_\theta^p=pZ_\theta.
	\]
	The process $Z$ is a nonnegative martingale, and
	$\mathbb E Z_\tau=1$ for every bounded stopping time $\tau$.
	
	\smallskip
	\noindent\emph{Basic bounds and spatial shape.}
	The restricted conjugate satisfies
	\[
	0\le H_+(q)\le R_+(0)q,
	\qquad q\ge0,
	\]
	because $h_+\ge0$, $h_+(0)=0$, and the conjugacy is restricted to
	$[0,R_+(0)]$.  Since $R_+$ is nondecreasing under
	Assumption~\ref{ass:dual-comparison}, every stopping rule $\tau\le s$
	has expected reward bounded above by $R_+(s)p$.  Waiting until the terminal
	time gives a nonnegative reward.  Hence
	\begin{equation}\label{eq:W-basic-bound}
		0\le\mathcal W(s,p)\le R_+(s)p,
		\qquad
		\mathcal W(s,0)=0.
	\end{equation}
	
	Let $0\le p_1<p_2$ and use the same stopping time for both initial values.
	The slopes of $H_+$ lie in $[0,R_+(0)]$, while the slope of
	$\Phi_+(r,\cdot)$ is $R_+(r)\le R_+(s)$ for $0\le r\le s$.  Since
	\[
	P_\theta^{p_2}-P_\theta^{p_1}
	=(p_2-p_1)Z_\theta,
	\]
	optional sampling yields
	\begin{equation}\label{eq:W-Lipschitz}
		0\le
		\mathcal W(s,p_2)-\mathcal W(s,p_1)
		\le R_+(s)(p_2-p_1).
	\end{equation}
	Thus $\mathcal W(s,\cdot)$ is nondecreasing and Lipschitz on compact
	$p$-intervals.  For a fixed stopping time, its expected reward is convex in
	$p$: the preterminal reward is affine, the terminal reward is convex, and
	$P^p$ depends linearly on $p$.  Taking the supremum over stopping times
	preserves convexity.
	
The reward in \eqref{eq:dual-obstacle-stopping} need not be continuous
at the terminal date.  It has instead the favorable upward jump
\[
\Phi_+(0,p)\le H_+(p),\qquad p\ge0.
\]
Continuity follows from the finite-horizon result
\cite[Theorem~3.1(i),(iii)]{PalczewskiStettner2010} after bounded
truncation, with the remaining horizon regarded as an external parameter.
More precisely, let
\[
\chi_N(z):=(-N)\vee(z\wedge N),
\qquad
\Phi_N:=\chi_N\circ\Phi_+,
\qquad
H_N:=\chi_N\circ H_+.
\]
The truncated rewards are bounded and continuous, and the terminal
majorization $\Phi_N(0,\cdot)\le H_N$ is preserved.  Hence their stopping
values $\mathcal W_N$ are jointly continuous, including at $s=0$.

It remains only to remove the truncation.  The rewards have a common
linear-growth bound and, for every $K<\infty$ and $a>1$,
\[
\sup_{0\le p\le K}
\mathbb E\!\left[
\sup_{0\le\theta\le S}(P_\theta^p)^a
\right]
\le
K^a\left(\frac{a}{a-1}\right)^a
e^{a(a-1)S}<\infty.
\]
Thus the family of stopped rewards is uniformly integrable, uniformly over
$p\in[0,K]$, $s\in[0,S]$, and all stopping times bounded by $s$.
Consequently,
\[
\sup_{(s,p)\in[0,S]\times[0,K]}
|\mathcal W_N(s,p)-\mathcal W(s,p)|
\longrightarrow0.
\]
It follows that
\[
\mathcal W\in C([0,S]\times[0,\infty)).
\]
	
	\smallskip
	\noindent\emph{An upper comparison function.}
	Define
	\begin{equation}\label{eq:upper-comparison-profile}
		\overline{\mathcal W}(s,p):=
		\begin{cases}
			\mathscr H_+(s,p)
			+\overline p_+(s)R_+(s)
			-\mathscr H_+(s,\overline p_+(s))-Q_+(s),
			&0\le p<\overline p_+(s),\\
			pR_+(s)-Q_+(s),&p\ge\overline p_+(s).
		\end{cases}
	\end{equation}
	By \eqref{eq:comparison-contact}, the two branches and their first spatial
	derivatives agree at $\overline p_+(s)$.  Convexity of
	$\mathscr H_+(s,\cdot)$ shows that the first branch dominates
	$pR_+-Q_+$, so $\overline{\mathcal W}\ge\Phi_+$.
	
	Using $(\mathscr H_+)_s=p^2(\mathscr H_+)_{pp}$ and differentiating the
	first branch at fixed $p$, the terms containing $\overline p_+'$ cancel by
	\eqref{eq:comparison-contact}.  Hence
	\[
	(\partial_s-p^2\partial_{pp})\overline{\mathcal W}
	=
	\begin{cases}
		\overline p_+R_+'-Q_+'-(\mathscr H_+)_s(s,\overline p_+),
		&p<\overline p_+(s),\\
		pR_+'-Q_+',&p>\overline p_+(s),
	\end{cases}
	\]
	which is nonnegative by Assumption~\ref{ass:dual-comparison}.
	No singular spatial measure occurs at the joining curve because the first
	spatial derivatives match there.
	Terminal-corner compatibility gives
	$\overline{\mathcal W}(0,\cdot)=H_+$.  Indeed,
	$\overline p_+(0)=h_+'(R_+(0))$: below this point the restricted and full
	conjugates agree, while above it both coincide with the supporting line
	$pR_+(0)-h_+(R_+(0))$ at $x=R_+(0)$.
	
	For a smooth comparison function $F$, It\^o's formula gives
	\[
	dF(s-\theta,P_\theta^p)
	=
	-\bigl(F_s-p^2F_{pp}\bigr)
	(s-\theta,P_\theta^p)\,d\theta+dM_\theta.
	\]
	Thus $\overline{\mathcal W}(s-\theta,P_\theta^p)$ is a supermartingale.
	The generalized It\^o formula gives the same conclusion for the
	piecewise-defined function because its first spatial derivative is
	continuous; hence no local-time term appears.  Optional sampling therefore
	yields
	\[
	\mathcal W(s,p)\le\overline{\mathcal W}(s,p).
	\]
	For $p\ge\overline p_+(s)$,
	\[
	\overline{\mathcal W}(s,p)
	=pR_+(s)-Q_+(s)=\Phi_+(s,p).
	\]
	Immediate stopping gives the reverse inequality, and hence
	\begin{equation}\label{eq:contact-above-comparison}
	\mathcal W(s,p)=\Phi_+(s,p),
	\qquad p\ge\overline p_+(s).
	\end{equation}
	
	\smallskip
	\noindent\emph{Threshold structure and bounds.}
	Set
	\[
	\Delta(s,p):=\mathcal W(s,p)-\Phi_+(s,p).
	\]
	Immediate stopping gives $\Delta\ge0$.  From
	\eqref{eq:W-Lipschitz}, for $p_2>p_1$,
	\[
	\Delta(s,p_2)-\Delta(s,p_1)
	\le0,
	\]
	so $p\mapsto\Delta(s,p)$ is nonincreasing.  Moreover,
	\[
	\Delta(s,0)=Q_+(s)>0,
	\]
	whereas \eqref{eq:contact-above-comparison} shows that $\Delta$ vanishes
	for all sufficiently large $p$.  Therefore there exists a unique
	$\beta(s)\in(0,\infty)$ such that
	\[
	\{\Delta(s,\cdot)=0\}=[\beta(s),\infty),
	\qquad
	\{\Delta(s,\cdot)>0\}=[0,\beta(s)).
	\]
	
	If $p<Q_+(s)/R_+(s)$, then
	\[
	\Phi_+(s,p)<0\le\mathcal W(s,p),
	\]
	so contact is impossible.  Together with
	\eqref{eq:contact-above-comparison}, this proves
	\[
	\frac{Q_+(s)}{R_+(s)}
	\le\beta(s)\le\overline p_+(s).
	\]
	The functions $Q_+/R_+$ and $\overline p_+$ are positive and continuous on
	$[0,S]$.  Continuity of $\overline p_+$ follows from
	$(\mathscr H_+)_{pp}>0$ and the implicit function theorem.  This proves the
	uniform bounds \eqref{eq:beta-uniform-bounds}.
	
	\smallskip
	\noindent\emph{Lower semicontinuity and the initial value.}
	Joint continuity of $\Delta$ makes the contact set
	\[
	\mathcal S
	=\{(s,p):p\ge\beta(s)\}
	\]
	closed.  Since this is the epigraph of $\beta$, the function $\beta$ is
	lower semicontinuous.
	
	At $s=0$, $\mathcal W(0,p)=H_+(p)$.  Strict convexity of $h_+$ gives
	\[
	H_+(p)>pR_+(0)-h_+(R_+(0))
	\quad\text{for }p<h_+'(R_+(0)),
	\]
	and equality for $p\ge h_+'(R_+(0))$.  Therefore
	\[
	\beta(0)=h_+'(R_+(0)).
	\]
	Lower semicontinuity gives
	\[
	\beta(0)\le\liminf_{s\downarrow0}\beta(s).
	\]
	On the other hand,
	\[
	\beta(s)\le\overline p_+(s),
	\qquad
	\overline p_+(s)\longrightarrow
	\overline p_+(0)=h_+'(R_+(0)).
	\]
	The corresponding limsup inequality proves
	\[
	\lim_{s\downarrow0}\beta(s)=h_+'(R_+(0)).
	\]
	
	\smallskip
	\noindent\emph{Continuation equation, value matching, and smooth fit.}
	The dynamic programming principle gives
	\begin{equation}\label{eq:dual-VI}
	\min\{\mathcal W-\Phi_+,
		\mathcal W_s-p^2\mathcal W_{pp}\}=0,
		\qquad p>0,\quad 0<s\le S,
	\end{equation}
	in the viscosity sense. Standard penalization and local parabolic Sobolev
	estimates show that it also holds almost everywhere locally away from
	$p=0$; see \cite{Friedman1975PVI}. On a compact cylinder contained in
	$\mathcal C$, the obstacle is inactive. After a logarithmic spatial
	change, the equation is uniformly parabolic, and interior regularity gives
	\[
	\mathcal W\in C_{\mathrm{loc}}^{1,2}(\mathcal C),
	\qquad
	\mathcal W_s-p^2\mathcal W_{pp}=0
	\quad\text{in }\mathcal C.
	\]
	
	Fix $s_0\in(0,S]$ and set $p_0:=\beta(s_0)>0$. In a neighborhood of
	$(s_0,p_0)$, let $U:=\mathcal W-\Phi_+$ and use the logarithmic coordinates
	$p=e^{z+s}$ introduced in
	Appendix~\ref{app:uniform-weighted-dual-second-derivative}. With
	$\widetilde U(s,z):=U(s,e^{z+s})$ and
	$\mathcal H:=\partial_{zz}-\partial_s$, the transformed variational
	inequality is
	\[
	\min\{\widetilde U,-\mathcal H\widetilde U+\widetilde f\}=0,
	\]
	where the locally H\"older source is
	\[
	\widetilde f(s,z)=e^{z+s}R_+'(s)-Q_+'(s).
	\]
	The same penalization and Sobolev estimates yield
	\[
	\widetilde U\in W_{\rho,\mathrm{loc}}^{2,1},
	\qquad
	\mathcal H\widetilde U
	=\widetilde f\mathbf 1_{\{\widetilde U\ne0\}}
	\quad\text{a.e.}
	\]
	for every finite $\rho>1$. Local Schauder theory supplies a function
	$v_0\in W_\infty^{2,1}$ satisfying $\mathcal Hv_0=\widetilde f$ in a
	smaller cylinder. Hence Theorem~1 of
	\cite{AnderssonLindgrenShahgholian2013} applies and gives
	$\widetilde U\in W_{\infty,\mathrm{loc}}^{2,1}$. In particular, its first
	spatial derivative is continuous across the contact set.
	
	Since $U(s_0,\cdot)\ge0$ and $U(s_0,p_0)=0$, spatial differentiability gives
	$U_p(s_0,p_0)=0$. Therefore
	\[
	\mathcal W_p(s_0,\beta(s_0)^-)
	=(\Phi_+)_p(s_0,\beta(s_0))
	=R_+(s_0).
	\]
	A backward cylinder gives the same conclusion when $s_0=S$. Value matching
	follows directly from $\beta(s)\in\mathcal S_s$:
	\[
	\mathcal W(s,\beta(s))
	=\Phi_+(s,\beta(s))
	=\beta(s)R_+(s)-Q_+(s).
	\]
	At $s=0$, the identities in \eqref{eq:w-free-boundary-problem} follow from
	$\mathcal W(0,\cdot)=H_+$, terminal-corner compatibility, and conjugacy:
	\[
	H_+(\beta(0))
	=\beta(0)R_+(0)-Q_+(0),
	\qquad
	H_+'(\beta(0))=R_+(0).
	\]
	Together with \eqref{eq:W-basic-bound}, this proves that
	$(\mathcal W,\beta)$ satisfies
	\eqref{eq:w-free-boundary-problem} with
	$(w,p_R)$ replaced by $(\mathcal W,\beta)$.
\end{proof}

	\section{Proof of Lemma~\ref{lem:uniform-weighted-dual-second-derivative}}
	\label{app:uniform-weighted-dual-second-derivative}

	For $z_*\in\mathbb R$, $s_*\in\mathbb R$, and $r>0$, write
	\[
	Q_r^-(z_*,s_*):=(z_*-r,z_*+r)\times(s_*-r^2,s_*],
	\]
	and let
	\[
	\mathcal H:=\partial_{zz}-\partial_s.
	\]

\begin{proof}
	Choose $p_*>0$ as in Proposition~\ref{prop:asymptotics_p0}. Then
	\begin{equation}\label{eq:appendix-near-zero-bound}
		\sup_{\substack{0\le s\le S\\0<p\le p_*}}
		p\mathcal W_{pp}(s,p)<\infty.
	\end{equation}
	Since $\beta(s)\le c^*$ by \eqref{eq:beta-uniform-bounds}, it remains to
	control the compact $p$-range
	\[
		K:=[0,S]\times[p_*/2,c^*].
	\]

	We first record the transformed problem, including its domain and data.
	Use the logarithmic change of variables and define the transformed gap
	directly by
	\begin{equation}\label{eq:appendix-unified-transform}
		p=e^{z+s},\qquad
		\widetilde U(s,z)
		:=\mathcal W(s,e^{z+s})-\Phi_+(s,e^{z+s}),\qquad
		\widetilde f(s,z):=e^{z+s}R_+'(s)-Q_+'(s).
	\end{equation}
	Thus the transformed space-time domain is
	\[
		\widetilde{\mathcal Q}:=(0,S]\times\mathbb R,
	\]
	and the set $K$ corresponds exactly to the compact slanted strip
	\begin{equation}\label{eq:appendix-transformed-compact-set}
		\widetilde K
		:=\left\{(s,z):0\le s\le S,\quad
		\log(p_*/2)-s\le z\le\log c^*-s\right\}.
	\end{equation}
	In particular, on $\widetilde K$,
	\[
		\log(p_*/2)-S\le z\le\log c^*.
	\]
	The degenerate boundary $p=0$ has been sent to $z=-\infty$ and therefore
	does not meet $\widetilde K$.

	Define the transformed free boundary and initial datum by
	\[
		\beta_{\widetilde U}(s):=\log\beta(s)-s,\qquad
		G_0(z):=H_+(e^z)-\Phi_+(0,e^z).
	\]
	Direct differentiation gives
	\[
		\widetilde U_s
		=\mathcal W_s-\widetilde f
		+p\bigl(\mathcal W_p-R_+(s)\bigr),
	\]
	\[
		\widetilde U_z
		=p\bigl(\mathcal W_p-R_+(s)\bigr),\qquad
		\widetilde U_{zz}
		=p\bigl(\mathcal W_p-R_+(s)\bigr)
		+p^2\mathcal W_{pp},
	\]
	where the derivatives of $\mathcal W$ are evaluated at $(s,p)$ with
	$p=e^{z+s}$. Therefore,
	\begin{equation}\label{eq:appendix-operator-identity}
		-\mathcal H\widetilde U+\widetilde f
		=\mathcal W_s-p^2\mathcal W_{pp}.
	\end{equation}
	By the dynamic programming argument in the proof of
	Proposition~\ref{prop:stopping-boundary}, $\mathcal W$ is a continuous
	viscosity solution of \eqref{eq:dual-VI}. The map
	$(s,z)\mapsto(s,e^{z+s})$ is a smooth diffeomorphism from
	$\widetilde{\mathcal Q}$ onto $(0,S]\times(0,\infty)$, and subtraction of
	the transformed obstacle,
	which is of class $C^{1,2}$, preserves the viscosity inequalities. Hence
	the invariance of viscosity solutions under smooth changes of variables,
	together with \eqref{eq:appendix-operator-identity}, shows that
	$\widetilde U$ is a continuous viscosity solution of the first equation
	below; see \cite{CrandallIshiiLions1992}. The initial condition and the
	continuation and stopping regions transform in the same way, giving
	\begin{equation}\label{eq:appendix-transformed-vi}
	\begin{cases}
		\min\{\widetilde U,-\mathcal H\widetilde U+\widetilde f\}=0,
		&(s,z)\in\widetilde{\mathcal Q},\\
		\widetilde U(0,z)=G_0(z),&z\in\mathbb R,\\
		\widetilde U(s,z)>0,&z<\beta_{\widetilde U}(s),\\
		\widetilde U(s,z)=0,&z\ge\beta_{\widetilde U}(s),\\
		\displaystyle\lim_{z\to-\infty}\widetilde U(s,z)=Q_+(s),
		&0\le s\le S.
	\end{cases}
	\end{equation}
	Spatial smooth fit becomes
	\begin{equation}\label{eq:appendix-transformed-smooth-fit}
		\widetilde U_z(s,\beta_{\widetilde U}(s)^-)=0.
	\end{equation}

	We next pass from the viscosity formulation to the strong Sobolev
	formulation. Fix backward cylinders
	$Q'\Subset Q\Subset\widetilde{\mathcal Q}$. Let
	$\gamma_\varepsilon\in C^\infty(\mathbb R)$ be a nonnegative smooth
	approximation of $r\mapsto\varepsilon^{-1}r^-$, and let
	$\widetilde U^\varepsilon$ solve
	\[
		-\mathcal H\widetilde U^\varepsilon+\widetilde f
		=\gamma_\varepsilon(\widetilde U^\varepsilon)
		\qquad\text{in }Q,
	\]
	with smooth approximations of the parabolic boundary data. The maximum
	principle and the standard penalty estimates give local bounds, independent
	of $\varepsilon$, for $\widetilde U^\varepsilon$ and its penalty term.
	The interior parabolic estimates therefore yield, for every finite
	$\rho>1$,
	\[
		\|\widetilde U^\varepsilon\|_{W_\rho^{2,1}(Q')}
		\le C_{Q',Q,\rho}.
	\]
	After passage to a subsequence, the penalized solutions converge weakly in
	$W_\rho^{2,1}(Q')$ and locally uniformly to a solution of the transformed
	variational inequality. Stability of viscosity solutions and the parabolic
	comparison principle for continuous viscosity sub- and supersolutions
	identify this limit with $\widetilde U$; see
	\cite{CrandallIshiiLions1992,Friedman1975PVI}. Since $Q'$ is arbitrary, for
	every finite $\rho>1$ we obtain
	\begin{equation}\label{eq:appendix-transformed-equation}
		\widetilde U\in W_{\rho,\mathrm{loc}}^{2,1}
		\quad\text{and}\quad
		\mathcal H\widetilde U
		=\widetilde f\mathbf 1_{\{\widetilde U\ne0\}}
		\quad\text{a.e. in }\widetilde{\mathcal Q}.
	\end{equation}
	Here the almost-everywhere identity is the complementary equation associated
	with \eqref{eq:appendix-transformed-vi}. The transformed estimate that
	remains to be proved is
	\begin{equation}\label{eq:appendix-transformed-goal}
		\|\widetilde U_{zz}\|_{L^\infty(\widetilde K)}<\infty.
	\end{equation}

	We now work entirely with \eqref{eq:appendix-transformed-vi}. At every point
	$(s_*,z_*)\in\widetilde K$ with $s_*>0$, choose a sufficiently small
	backward cylinder $Q_r^-(z_*,s_*)$. This also covers $s_*=S$. Define
	\[
		\widetilde v_0(s,z):=Q_+(s)-e^{z+s}R_+(s).
	\]
	The regularity of $R_+$ and $Q_+$ gives
	$\widetilde v_0\in W_{\infty,\mathrm{loc}}^{2,1}$, and direct
	differentiation yields
	\[
		\mathcal H\widetilde v_0
		=e^{z+s}R_+'(s)-Q_+'(s)=\widetilde f(s,z).
	\]
	Thus the source hypothesis in Theorem~1 of
	\cite{AnderssonLindgrenShahgholian2013} is satisfied. Applying that theorem
	to \eqref{eq:appendix-transformed-equation} therefore gives
	\[
		\widetilde U\in W_{\infty,\mathrm{loc}}^{2,1}
	\]
	near every positive-time point of $\widetilde K$.

	It remains to consider the initial face $s=0$. Set
	\[
		z_0:=\beta_{\widetilde U}(0)=\log h_+'(R_+(0)).
	\]
	Since $\beta_{\widetilde U}(s)\to z_0$ as $s\downarrow0$, the
	continuation/stopping
	description in \eqref{eq:appendix-transformed-vi} shows that points $(0,z_*)$ with
	$z_*<z_0$ have a relative neighborhood in the continuation region, whereas
	$\widetilde U=0$ in a relative neighborhood of every $(0,z_*)$ with
	$z_*>z_0$. Standard linear parabolic regularity handles both cases. Thus
	only the corner $(0,z_0)$ requires an extension argument.

	By conjugacy, the transformed initial datum satisfies
	\[
		G_0>0\quad\text{on }(-\infty,z_0),\qquad
		G_0=0\quad\text{on }[z_0,\infty),\qquad
		G_0(z_0)=G_0'(z_0)=0.
	\]
	Moreover, $h_+''(R_+(0))>0$ implies
	$G_0\in W_{\mathrm{loc}}^{2,\infty}$ near $z_0$; the equality
	$G_0'(z_0^-)=G_0'(z_0^+)=0$ excludes a Dirac mass in $G_0''$.

	Choose $r,\delta>0$ so that the positive-time part of
	\[
		\mathcal Q_0:=(-\delta,\delta)\times(z_0-r,z_0+r)
	\]
	lies in a neighborhood where \eqref{eq:appendix-transformed-equation}
	holds, and define
	\[
	\overline U(s,z):=
	\begin{cases}
		G_0(z),&s<0,\\
		\widetilde U(s,z),&s\ge0,
	\end{cases}
	\qquad
	\overline f(s,z):=
	\begin{cases}
		G_0''(z),&s<0,\\
		\widetilde f(s,z),&s>0.
	\end{cases}
	\]
	The value of $\overline f$ at $s=0$ is immaterial. The time traces agree
	because $\widetilde U(0,\cdot)=G_0$, and the preceding derivative matching
	excludes a spatial atom at $z_0$. Since $G_0''=0$ almost everywhere on
	$\{G_0=0\}$, we obtain
	\begin{equation}\label{eq:appendix-extended-equation}
		\mathcal H\overline U
		=\overline f\mathbf 1_{\{\overline U\ne0\}}
		\qquad\text{in }\mathcal D'(\mathcal Q_0).
	\end{equation}
	Here $\mathcal D'(\mathcal Q_0)$ denotes the space of distributions on
	$\mathcal Q_0$; thus \eqref{eq:appendix-extended-equation} is understood in
	the distributional sense.
	The right-hand side is locally bounded, so standard parabolic estimates
	yield
	\[
		\overline U\in W_{\rho,\mathrm{loc}}^{2,1}(\mathcal Q_0)
		\qquad\text{for every finite }\rho>1;
	\]
	see \cite{LadyzhenskayaSolonnikovUralceva1968,Lieberman1996}.

	To verify the source condition in Theorem~1 of
	\cite{AnderssonLindgrenShahgholian2013}, choose a spatial cutoff $\chi$
	equal to one near $z_0$, extend $G_0$ to a function
	$\widehat G_0\in W^{2,\infty}(\mathbb R)$ that agrees with it there, and set
	\[
		\widehat f(s,z):=\chi(z)\widetilde f(s,z),
		\qquad s\ge0.
	\]
	Let $(P_s)_{s\ge0}$ be the heat semigroup and define the transformed
	comparison function
	\[
	\widetilde v(s,z):=
	\begin{cases}
		\widehat G_0(z),&s<0,\\[1mm]
		P_s\widehat G_0(z)-
		\displaystyle\int_0^s
		P_{s-\theta}\widehat f(\theta,\cdot)(z)\,d\theta,
		&s\ge0.
	\end{cases}
	\]
	Its time traces agree at $s=0$, and the sign in the Duhamel term gives
	$\mathcal H\widetilde v=\overline f$ in a smaller cylinder around
	$(0,z_0)$. The local H\"older regularity of $\widetilde f$ and standard
	heat-semigroup estimates give
	$\widetilde v\in W_\infty^{2,1}$ there. Hence the source condition is
	satisfied, and the cited theorem applied to
	\eqref{eq:appendix-extended-equation} yields
	\[
		\overline U\in W_{\infty,\mathrm{loc}}^{2,1}
	\]
	near $(0,z_0)$.

	We have now obtained local $W_\infty^{2,1}$ estimates for $\widetilde U$ in
	a relative neighborhood of every point of the compact set $\widetilde K$.
	A finite-cover argument proves \eqref{eq:appendix-transformed-goal}.

	Finally, return to $p=e^{z+s}$. Since
	\[
		\widetilde U_{zz}
		=p(\mathcal W_p-R_+(s))+p^2\mathcal W_{pp},
	\]
	we have on $K$
	\begin{equation}\label{eq:appendix-transform-back}
		p\mathcal W_{pp}
		=\frac{\widetilde U_{zz}}{p}-\bigl(\mathcal W_p-R_+(s)\bigr).
	\end{equation}
	Here $p\ge p_*/2$, $0\le\mathcal W_p\le R_+(s)$, and $R_+$ is bounded.
	Thus \eqref{eq:appendix-transformed-goal} bounds
	$p\mathcal W_{pp}$ on $K$. Combining this with
	\eqref{eq:appendix-near-zero-bound} proves
	\eqref{eq:uniform-weighted-dual-second-derivative}.
	\end{proof}

\section*{Acknowledgments}
This work was supported in part by the National Natural Science Foundation
of China (NSFC) under Grant Nos.~12271274 and 12571514.

\bibliographystyle{abbrvnat}
\bibliography{references}

\end{document}